\documentclass[reqno]{amsart}
\usepackage{amsmath}
\usepackage{amssymb}
\usepackage{amsfonts}
\usepackage{amsthm}

\usepackage[T1]{fontenc}
\usepackage[sc]{mathpazo}
\usepackage[mathcal]{euscript}

\usepackage{subfigure}
\usepackage{tikz}
\usepackage{xcolor}
\definecolor{halfgray}{gray}{0.55}
\definecolor{OliveGreen}{rgb}{0,.35,0}
\definecolor{webbrown}{rgb}{.6,0,0}
\definecolor{BrightViolet}{rgb}{0.5,0.2,0.8}
\definecolor{Maroon}{cmyk}{0, 0.87, 0.68, 0.32}
\definecolor{RoyalBlue}{cmyk}{1, 0.50, 0, 0.25}
\definecolor{Black}{cmyk}{0, 0, 0, 0}

\usepackage{paralist}
\usepackage{xspace}

\usepackage[sort&compress,longnamesfirst]{natbib}
\usepackage{hyperref}
\hypersetup{
colorlinks=true,
linktocpage=true,
pdfstartview=FitH,
breaklinks=true,
pdfpagemode=UseNone,
pageanchor=true,
pdfpagemode=UseOutlines,
plainpages=false,
bookmarksnumbered,
bookmarksopen=false,
bookmarksopenlevel=1,
hypertexnames=true,
pdfhighlight=/O,
urlcolor=Maroon, linkcolor=RoyalBlue, citecolor=OliveGreen, % <--- for screen
pdftitle={},
pdfauthor={},
pdfsubject={},
pdfkeywords={},
pdfcreator={pdfLaTeX},
pdfproducer={LaTeX with hyperref}
}

\newcommand{\R}{\mathbb{R}}

\DeclareMathOperator{\bd}{bd}
\DeclareMathOperator{\bigoh}{\mathcal O}

\DeclareMathOperator{\exclude}{\backslash}

\DeclareMathOperator{\Int}{int}

\DeclareMathOperator{\midd}{\|}

\DeclareMathOperator{\simplex}{\Delta}
\DeclareMathOperator{\supp}{supp}

\newcommand{\eps}{\varepsilon}
\newcommand{\from}{\colon}

\newcommand{\txs}{\textstyle}

\newcommand{\insum}{\sum\nolimits}
\newcommand{\inprod}{\prod\nolimits}

\newcommand{\set}{\mathcal{S}}
\newcommand{\play}{\mathcal{N}}
\newcommand{\act}{\mathcal{A}}
\newcommand{\strat}{X}

\newcommand{\game}{\mathfrak{G}}

\usepackage[multiple]{footmisc}

\DeclareMathOperator{\sign}{sgn}

\newcommand{\dd}{\;d}
\newcommand{\id}{\dd}
\newcommand{\ddt}[1][]{\frac{d^{#1}}{dt^{#1}}}

\newcommand{\flow}{\Theta}
\newcommand{\phase}{\Omega}

\newcommand{\intstrat}{\Int(\strat)}

\newcommand{\cone}{T^{c}}
\newcommand{\dkl}{D_{\textup{KL}}}
\newcommand{\pf}{\mathcal{Z}}
\newcommand{\sw}{$\mathbf{S}^{\infty}\mathbf{W}$\xspace}
\newcommand{\neginf}{\boldsymbol{-\infty}}

\theoremstyle{plain}
\newtheorem{theorem}{Theorem}
\newtheorem{corollary}[theorem]{Corollary}
\newtheorem*{corollary*}{Corollary}
\newtheorem{lemma}[theorem]{Lemma}
\newtheorem{proposition}[theorem]{Proposition}
\newtheorem*{proposition*}{Proposition}

\theoremstyle{definition}

\newtheorem*{definition*}{Definition}

\theoremstyle{remark}
\newtheorem{remark}{Remark}
\newtheorem*{remark*}{Remark}

\numberwithin{equation}{section}
\numberwithin{theorem}{section}
\begin{document}

%!TEX TS-program =  pdflatex

%*************************************************************
%*****    FRONT MATTER
%*************************************************************
%\begin{frontmatter}

%----------------------------------------------------------------------
%%% TITLE & AUTHORS
%----------------------------------------------------------------------
\title{Higher Order Game Dynamics}

\author{Rida Laraki}
\address{French National Center for Scientific Research (CNRS) and \'Ecole Polytechnique, Paris, France}
\email{rida.laraki@polytechnique.edu}
\urladdr{\url{https://sites.google.com/site/ridalaraki/}}

\author[Panayotis Mertikopoulos]{Panayotis Mertikopoulos$^{\ast}$}
\address{French National Center for Scientific Research (CNRS) and Laboratoire d'Informatique de Gre\-no\-ble, Grenoble, France}
\email{panayotis.mertikopoulos@imag.fr}
\urladdr{\url{http://mescal.imag.fr/membres/panayotis.mertikopoulos/home.html}}

\thanks{$^{\ast}$Corresponding author.}
\thanks{\quad
Part of this work was carried out when Panayotis Mertikopoulos was with the Economics Department of the \'Ecole Polytechnique, Paris, France, and we would like to thank the \'Ecole for its hospitality and generous support.
The authors also gratefully acknowledge financial support from the French National Agency for Research (grant ANR-10-BLAN 0112-JEUDY), the GIS ``Sciences de la D\'ecision'', and the P\^ole de Recherche
en Math\'ematiques, Sciences, et Technologies de l'Information et de la Communication under grant no. C-UJF-LACODS MSTIC 2012.
%MSTIC under grant no. C-UJF-LACODS-2012.
}

\begin{abstract}
Continuous-time dynamics for games are typically first order systems where payoffs determine the growth rate of the players' strategy shares.
In this paper, we investigate what happens beyond first order by viewing payoffs as higher order forces of change, specifying e.g. the acceleration of the players' evolution instead of its velocity (a viewpoint which emerges naturally when it comes to aggregating empirical data of past instances of play).
To that end, we derive a wide class of higher order game dynamics, generalizing first order imitative dynamics, and, in particular, the replicator dynamics.
We show that strictly dominated strategies become extinct in $n$-th order payoff-monotonic dynamics $n$ orders as fast as in the corresponding first order dynamics;
furthermore, in stark contrast to first order, weakly dominated strategies also become extinct for $n\geq2$.
All in all, higher order payoff-monotonic dynamics lead to the elimination of weakly dominated strategies, followed by the iterated deletion of strictly dominated strategies, thus providing a dynamic justification of the well-known epistemic rationalizability process of \cite{DF90}.
Finally, we also establish a higher order analogue of the folk theorem of evolutionary game theory, and we show that convergence to strict equilibria in $n$-th order dynamics is $n$ orders as fast as in first order.
\end{abstract}

%----------------------------------------------------------------------
%%% KEYWORDS
%----------------------------------------------------------------------


%\begin{keyword}
%Game dynamics
%\sep
%higher order dynamical systems
%\sep
%(weakly) dominated strategies
%\sep
%folk theorem
%\sep
%learning
%\sep
%replicator dynamics
%\sep
%stability of equilibria.
%\end{keyword}
%

%\end{frontmatter}

\maketitle

%*************************************************************
%*****    BODY TEXT
%*************************************************************

%----------------------------------------------------------------------
%%% INTRODUCTION
%----------------------------------------------------------------------
\section{Introduction.}
\label{sec.introduction}

Owing to the considerable complexity of computing Nash equilibria and other rationalizable outcomes in non-cooperative games, a fundamental question that arises is whether these outcomes may be regarded as the result of a dynamic learning process where the participants ``accumulate empirical information on the relative advantages of the various pure strategies at their disposal'' \citep[p.~21]{NashThesis}.
To that end, numerous classes of game dynamics have been proposed (from both a learning and an evolutionary ``mass-action'' perspective), each with its own distinct set of traits and characteristics \textendash\ see e.g. the comprehensive survey by \cite{San10} for a most recent account.

Be that as it may, there are few rationality properties that are shared by a decisive majority of game dynamics.
%For instance, in the case of learning in finite $n$-person games (the focus of our paper), the convergence of fictitious play and its variants is typically in an empirical, time-averaged sense (see e.g. \citealp{FL98}), whereas reinforcement approaches yield convergence of mixed strategies to approximate/perturbed solution concepts \citep{LC05,CoMeSo10}.
%On the other hand, if we regard ``evolution'' as ``learning in large populations of myopic agents'', and focusing for simplicity on the continuous-time deterministic regime, then the situation is not particularly more uniform either:
For instance, if we focus on the continuous-time, deterministic regime, a simple comparison between the well-known replicator dynamics \citep{TJ78} and the Smith dynamics \citep{Smi84} reveals that game dynamics can be imitative (replicator) or innovative (Smith),
rest points might properly contain the game's Nash set or coincide with it \citep{HS09},
and strictly dominated strategies might become extinct \citep{SZ92} or instead survive \citep{HS11}.
In fact, negative results seem to be much more ubiquitous:
there is no class of uncoupled game dynamics that always converges to equilibrium \citep{HMC03} and weakly dominated strategies may survive in the long run, even in simple $2\times2$ games \citep{Sam93,Wei95}.

From a mathematical standpoint, the single unifying feature of the vast majority of game dynamics is that they are first order dynamical systems.
%\textendash\ even continuity is absent in several important classes of dynamics, such as the projection dynamics of \cite{SDL08}.
Interestingly however, this restriction to first order is not present in the closely related field of optimization (corresponding to games against nature):
as it happens, the second order ``heavy ball with friction'' method studied by \cite{Alv00} and \cite{AGR00} has some remarkable optimization properties that first order schemes do not possess.
In particular, by interpreting the gradient of the function to be maximized as a physical, Newtonian force (and not as a first order vector field to be tracked by the system's trajectories), one can give the system enough energy to escape the basins of attraction of local maxima and converge instead to the \emph{global} maximum of the objective function (something which is not possible in ordinary first order dynamics).

\smallskip

This, therefore, begs the question:
\emph{can second (or higher) order dynamics be introduced and justified in a game theoretic setting?}
%\emph{Is there a context in which they arise naturally?}
And if yes,
\emph{do they allow us to obtain better convergence results and/or escape any of the first order impossibility results?}

\smallskip

The first challenge to overcome here is that second order methods in optimization apply to \emph{unconstrained} problems, whereas game dynamics must respect the (constrained) structure of the game's strategy space.
To circumvent this constraint, \cite{FM04} proposed a heavy-ball method as in \cite{AGR00} above, and they enforced consistency by projecting the orbits' velocity to a subspace of admissible directions when the updating would lead to inadmissible strategy profiles (say, assigning negative probability to an action).
Unfortunately, as is often the case with projection-based schemes (see e.g. \citealp{SDL08}), the resulting dynamics are not continuous, so even basic existence and uniqueness results are hard to obtain.

On the other hand, if players try to improve their performance by aggregating information on the relative payoff differences of their pure strategies, then this cumulative empirical data is not constrained (as mixed strategies are).
Thus, a promising way to obtain a well-behaved second order dynamical system for learning in games is to use the player's accumulated data to define an unconstrained performance measure for each strategy (this is where the dynamics of the process come in), and then map these ``scores'' to mixed strategies by means e.g. of a logit choice model \citep{Rus99,HSV09,MM10,Sor09}.
In other words, the dynamics can first be specified on an unconstrained space, and then mapped to the game's strategy space via the players' choice model.

This use of aggregate performance estimates also has important implications from the point of view of evolutionary game theory and population dynamics.
Indeed, it is well-known that the replicator dynamics arise naturally in populations of myopic agents that evolve based on ``imitation of success'' \citep{Hof95,San10} or on ``imitation driven by dissatisfaction'' \citep{BW96}.
Revision protocols of this kind are invariably steered by the players' \emph{instantaneous} payoffs;
remarkably however, if players are more sophisticated and keep an aggregate (or average) of their payoffs over time, then the same revision rules driven by \emph{long-term} success or dissatisfaction give rise to the same higher order dynamics discussed above.

\subsubsection*{Paper outline.}

After a few preliminaries in Section \ref{sec.prelims}, we make this approach precise in Section \ref{sec.dynamics}, where we derive a higher order variant of the well-known replicator dynamics of \cite{TJ78}.
%class of higher order dynamics extending the familiar \emph{imitative dynamics} of \citet{BW96} \textendash\   a class which contains all payoff-monotonic dynamics, and in particular, the replicator dynamics.
%In fact, the passage from performance scores to mixed strategies naturally induces a game-independent ``adjustment'' term which slows down the orbits that approach the boundary of the game's strategy space and renders the latter invariant.
Regarding the rationality properties of the derived dynamics, we show in Section \ref{sec.dominance} that the higher order replicator dynamics eliminate strictly dominated strategies, including iteratively dominated ones:
in the long run, only iteratively undominated strategies survive.
Qualitatively, this result is the same as its first order counterpart;
quantitatively however, the rate of extinction increases dramatically with the order of the dynamics:
\emph{dominated strategies become extinct in the $n$-th order replicator dynamics $n$ orders as fast as in first order}
(Theorem \ref{thm.dom}).

The reason for this enhanced rate of elimination is that empirical data accrues much faster if a higher order scheme is used rather than a lower order one:
%\footnote{For instance, $\dot x(t) = 1$ grows linearly in $t$, while the growth of $\ddot x(t) = 1$ is quadratic.}
players who use a higher order learning rule end up looking deeper into the past, so they identify consistent payoff differences and annihilate dominated strategies much faster.
As a consequence of the above, in the higher order ($n\geq2$) replicator dynamics, \emph{even weakly dominated strategies become extinct} (Theorem \ref{thm.dom.weak}, a result which comes in stark contrast to the first order setting.
%\textendash\ where weakly dominated strategies may survive even in simple $2\times2$ games such as Entry Deterrence \citep[Ex.~5.4]{Wei95}.
The higher order replicator dynamics thus perform one round of elimination of weakly dominated strategies followed by the iterated elimination of strictly dominated strategies;
from an epistemic point of view, \cite{DF90} showed that the outcome of this deletion process is all that can be expected from rational players who are not certain of their opponents' payoffs, so our result may be regarded as a dynamic justification of this form of rational behavior.

Extending our analysis to equilibrium play, we show in Section \ref{sec.folk} that modulo certain technical modifications, the folk theorem of evolutionary game theory \citep{HS88,Wei95} continues to hold in our higher order setting.
More specifically, we show that:
\begin{inparaenum}[\itshape a\upshape)]
\item
if an interior solution orbit converges, then its limit is Nash;
\item
if a point is Lyapunov stable, then it is also Nash; and
\item
if players start close enough to a strict equilibrium and with a small learning bias, then they converge to it;
conversely, only strict equilibria have this property (Theorem \ref{thm.folk}).
\end{inparaenum}
In fact, echoing our results on the rate of extinction of dominated strategies, we show that the $n$-th order replicator dynamics converge to strict equilibria $n$ orders as fast as in first order.

Finally, in Section \ref{sec.extensions}, we consider a much wider class of higher order dynamics that extends the familiar \emph{imitative dynamics} of \citet{BW96} \textendash\ including all payoff-monotonic dynamics \citep{SZ92} and, in particular, the replicator dynamics.
The results that we described above go through essentially unchanged for all higher order payoff-monotonic dynamics, with one notable trait standing out:
the property that only pure strategy profiles can be attracting holds in \emph{all} higher order imitative dynamics for $n\geq2$, and not only for the $n$-th order replicator dynamics.
As with the elimination of weakly dominated strategies, this is not the case in first order:
for instance, the payoff-adjusted replicator dynamics of Maynard Smith exhibit interior attractors even in simple $2\times2$ games (see e.g. Ex.~5.3 in \citealp{Wei95}).
\section{Notation and preliminaries.}
\label{sec.prelims}

\subsection{Notational conventions.}
\label{sec.notation}

If $\set=\{s_{\alpha}\}_{\alpha=0}^{n}$ is a finite set, the vector space spanned by $\set$ over $\R$ will be the set $\R^{\set}$ of all maps $x\from\set\to\R$, $s\in\set\mapsto x_{s}\in\R$.
The canonical basis $\{e_{s}\}_{s\in\set}$ of this space consists of the indicator functions $e_{s}\from\set\to\R$ which take the value $e_{s}(s)=1$ on $s$ and vanish otherwise, so, thanks to the identification $s\mapsto e_{s}$, we will not distinguish between $s\in\set$ and the corresponding basis vector $e_{s}$ of $\R^{\set}$.
In the same spirit,
%to avoid drowning in a morass of indices,
we will use the index $\alpha$ to refer interchangeably to either $s_{\alpha}$ or $e_{\alpha}$ (writing e.g. $x_{\alpha}$ instead of $x_{s_{\alpha}}$);
likewise, if $\{\set_{k}\}_{k\in\mathcal{K}}$ is a finite family of finite sets indexed by $k\in\mathcal{K}$, we will write $(\alpha;\alpha_{-k})$ for the tuple $(\alpha_{0},\dotsc,\alpha_{k-1},\alpha,\alpha_{k+1},\dotsc) \in\prod_{k}\set_{k}$ and $\sum_{\alpha}^{k}$ in place of $\sum_{\alpha\in\set_{k}}$.

We will also identify the set $\simplex(\set)$ of probability measures on $\set$ with the $n$-di\-men\-sio\-nal simplex of $\R^{\set}$: $\simplex(\set) \equiv \{x\in \R^{\set}: \sum_{\alpha} x_{\alpha} =1 \text{ and }x_{\alpha}\geq 0\}$.
Finally, regarding players and their actions, we will follow the original convention of Nash and employ Latin indices ($j,k,\dotsc$) for players, while keeping Greek ones ($\alpha,\beta,\dotsc$) for their actions (pure strategies); also, unless otherwise mentioned, we will use $\alpha,\beta,\dotsc$, for indices that start at $0$, and $\mu,\nu,\dotsc$, for those which start at 1.

\subsection{Finite games.}

A \emph{finite game in normal form} will comprise a finite set of \emph{players} $\play = \{1,\dotsc,N\}$, each with a finite set of \emph{actions} (or \emph{pure strategies}) $\act_{k} = \{\alpha_{k,0}, \alpha_{k,1},\dotsc\}$ that can be mixed by means of a probability distribution (\emph{mixed strategy} $x_{k} \in\simplex(\act_{k})$.
The set $\simplex(\act_{k})$ of a player's mixed strategies will be denoted by $\strat_{k}$, and aggregating over all players, the space of \emph{strategy profiles} $x = (x_{1}, \dotsc, x_{N})\in\prod_{k}\R^{\act_{k}}$ will be the product $\strat\equiv\inprod_{k} \strat_{k}$;
in this way, if $\act = \coprod_{k}\act_{k}$ denotes the (disjoint) union of the players' action sets, $\strat$ may be seen as a product of simplices embedded in $\R^{\act} \cong \prod_{k} \R^{\act_{k}}$.

As is customary, when we wish to focus on the strategy of a specific (focal) player $k\in\play$ versus that of his \emph{opponents} $\play_{-k} \equiv \play \exclude \{k\}$, we will use the shorthand $(x_{k};x_{-k}) \equiv (x_{1}, \dotsc, x_{k}, \dotsc, x_{N})\in \strat$ to denote the strategy profile where player $k$ plays $x_{k}\in\strat_{k}$ against the strategy $x_{-k}\in\strat_{-k} \equiv \prod_{\ell\neq k} \strat_{\ell}$ of his opponents.
The players' (expected) rewards are then prescribed by the game's \emph{payoff} (or \emph{utility} \emph{functions} $u_{k}\from\strat\to\R$:
\begin{equation}
\label{eq.payoff}
u_{k}(x) = \insum_{\alpha_{1}}^{1}\dotsi \insum_{\alpha_{N}}^{N} u_{k} (\alpha_{1},\dotsc,\alpha_{N})\,
x_{1,\alpha_{1}} \!\dotsm\, x_{N,\alpha_{N}},
\end{equation}
where $u_{k} (\alpha_{1},\dotsc,\alpha_{N})$ denotes the reward of player $k$ in the profile $(\alpha_{1},\dotsc,\alpha_{N})\in\prod_{k}\act_{k}$;
specifically, if player $k$ plays $\alpha\in\act_{k}$, we will use the notation:
\begin{equation}
\label{eq.actpayoff}
u_{k\alpha}(x)
\equiv u_{k}(\alpha;x_{-k})
= u_{k}(x_{1},\dotsc, \alpha, \dotsc, x_{N}).
\end{equation}

In light of the above, a game in normal form with players $k\in\play$, action sets $\act_{k}$ and payoff functions $u_{k}\from\strat\to\R$ will be denoted by $\game\equiv\game(\play,\act,u)$.
A \emph{restriction} $\game'$ of $\game$ (denoted $\game'\leq\game$) will then be a game $\game'\equiv\game'(\play,\act',u')$ played by the players of $\game$, each with a subset $\act_{k}'\subseteq\act_{k}$ of their original actions, and with payoff functions $u_{k}'\equiv u_{k}|_{\strat'}$ suitably restricted to the reduced strategy space $\strat' = \prod_{k}\simplex(\act_{k}')$ of $\game'$.

Given a game $\game\equiv\game(\play,\act,u)$, we will say that the pure strategy $\alpha\in\act_{k}$ is (\emph{strictly}) \emph{dominated} by $\beta\in\act_{k}$ (and we will write $\alpha\prec\beta$) when
\begin{equation}
\label{eq.dominated.pure}
u_{k\alpha}(x) < u_{k\beta}(x)\quad
\text{for all strategy profiles $x\in\strat$.}
\end{equation}
More generally, we will say that \emph{$q_{k}\in\strat_{k}$ is dominated by $q_{k}'\in\strat_{k}$} if
\begin{equation}
\label{eq.dominated.mixed}
u_{k}(q_{k};x_{-k}) < u_{k}(q_{k}';x_{-k})\;
\text{ for all strategies $x_{-k}\in\strat_{-k}$ of $k$'s opponents.}
\end{equation}
Finally, if the above inequalities are only strict for some (but \emph{not all}) $x\in\strat$, then we will employ the term \emph{weakly dominated} and write $q_{k}\preccurlyeq q_{k}'$ instead.

Of course, by removing dominated (and, thus, rationally unjustifiable) strategies from a game $\game$, other strategies might become dominated in the resulting restriction of $\game$, leading inductively to the notion of \emph{iteratively dominated strategies}:
%Specifically, given two subsets $M_{k}$ and $M_{-k}$ of $\strat_{k}$ and $\strat_{-k}$ respectively, let
%\begin{equation}
%\just(M_{k},M_{-k}) \equiv \{q_{k}\in M_{k}: \text{$\forall q_{k}' \in M_{k},\exists q_{-k}\in M_{-k}$ s.t. $u_{k}(q_{k};q_{-k}) \geq u_{k}(q_{k}';q_{-k})$}\}
%\end{equation}
%be the set of strategies $q_{k}\in M_{k}$ that are \emph{justifiable} (i.e. not dominated) with respect to any strategy $q_{-k}\in M_{-k}$.
%Then, starting with $\strat_{k}^{0} \equiv \strat_{k}$, define inductively the set of strategies that survive $r$ elimination rounds as $\strat_{k}^{r} = \just(\strat_{k}^{r-1},\strat_{-k}^{r-1})$ where $\strat_{-k}^{r-1} \equiv \prod_{\ell\neq k} \strat_{\ell}^{r-1}$;
%similarly, the \emph{pure} strategies that survive after $r$ rounds will be denoted by $\act_{k}^{r} \equiv \act_{k}\cap\strat_{k}^{r}$.
%In this way, the sequence $\{\strat_{k}^{r}\}_{r=0}^{\infty}$ forms a descending chain $\strat_{k}^{0}\supseteq\strat_{k}^{1}\supseteq\dotsc$ whose limit $\strat_{k}^{\infty}\equiv\intersect_{r=0}^{\infty}\strat_{k}^{r}$ consists of those strategies of player $k$ that are \emph{iteratively undominated}, i.e. they survive all rounds of elimination.
%In particular, if the space $\strat^{\infty}$ of iteratively undominated strategies is a singleton, $\game$ will be called \emph{dominance-solvable} and the sole surviving strategy in $\strat^{\infty}$ will be the game's \emph{rational solution}.
specifically, if a strategy survives all rounds of elimination, then it will be called \emph{iteratively undominated}, and if the space $\strat^{\infty}$ of iteratively undominated strategies is a singleton, the game $\game$ will be called \emph{dominance-solvable}.
%Moreover, if play evolves over time, say along the path $x(t)\in\strat$, $t\geq0$, we will say that a pure strategy $\alpha\in\act_{k}$ \emph{becomes extinct along $x(t)$} if $x_{k\alpha}(t) \to 0$ as $t \to\infty$;
%more generally, for mixed strategies $q_{k}\in\strat_{k}$, we will follow \cite{SZ92} and say that $q_{k}$ becomes extinct along $x(t)$ if $\min\{x_{k\alpha}(t):\alpha\in\supp(q_{k})\}\to0$, with the minimum taken over the support $\supp(q_{k}) \equiv \{\beta\in\act_{k}:q_{k\beta}>0\}$ of $q_{k}$.

On the other hand, when a game cannot be solved by removing dominated strategies, we will turn to the equilibrium concept of Nash which characterizes profiles that are resilient against unilateral deviations;
formally, we will say that $q\in\strat$ is a \emph{Nash equilibrium} of $\game$ if
\begin{equation}
\label{eq.Nash}
u_{k} (x_{k};q_{-k}) \leq u_{k}(q)\quad
\text{for all $x_{k}\in\strat_{k}$ and for all $k\in\play$.}
\end{equation}
If \eqref{eq.Nash} is strict for all $x_{k}\in\strat_{k}\exclude\{q_{k}\}$, $k\in\play$, $q$ itself will be called \emph{strict};
finally, equilibria of restrictions $\game'$ of $\game$ will be called \emph{restricted equilibria} of $\game$.

\subsection{Dynamical systems.}
\label{sec.prelims.dynamics}

Following \cite{Lee03}, a \emph{flow} on $\strat$ will be a smooth map $\flow\from \strat\times\R_{+}\to\strat$ such that
\begin{inparaenum}[\itshape a\upshape)]
\item
$\flow(x,0) = x$ for all $x\in \strat$;
and
\item
$\flow(\flow(x,t),s) = \flow(x,t+s)$ for all $x\in \strat$ and for all $s,t\geq 0$.
\end{inparaenum}
The curve $\flow^{x}\from\R_{+}\to\strat$, $t\mapsto \flow(x,t)$, will be called the \emph{orbit} (or \emph{trajectory}) of $x$ under $\flow$, and when there is no danger of confusion, $\flow^{x}(t)$ will be denoted more simply by $x(t)$.
In this way, $\flow$ induces a vector field $V$ on $\strat$ via the mapping $x\mapsto V(x) \equiv \dot x(0) \in \cone_{x}\strat$ where $\dot x(0)$ is the initial velocity of $x(t)$ and $\cone_{x}\strat$ denotes the tangent cone to $\strat$ at $x$, viz.:
\begin{equation}
\label{eq.tangent}
\cone_{x}\strat\equiv
	\{z\in\R^{\act}: \insum_{\alpha}^{k} z_{k\alpha} = 0 \text{ for all $k\in\play$ and } z_{k\alpha} \geq0 \text{ if $x_{k\alpha}=0$}\}.
\end{equation}
By the fundamental theorem on flows, $x(t)$ will be the unique solution to the (first order) \emph{dynamical system} $\dot x(t) = V(x(t))$, $t\geq0$.
Accordingly, we will say that $q\in\strat$ is:
\begin{itemize}
\setlength{\itemsep}{2pt}
\setlength{\parskip}{1pt}
\item
\emph{stationary} if $V(q) = 0$ (i.e. if $q(t) \equiv \flow(q,t) = q$ for all $t\geq0$).
\item
\emph{Lyapunov stable} if, for every neighborhood $U$ of $q$, there exists a neighborhood $V$ of $q$ such that $x(t)\in U$ for all $x\in V$, $t\geq0$.
\item
\emph{attracting} if $x(t)\to q$ for all $x$ in a neighborhood $U$ of $q$ in $\strat$.
\item
\emph{asymptotically stable} if it is Lyapunov stable and attracting.
\end{itemize}

Higher order dynamics of the form ``$x^{(n)} = V$'' are defined via the recursive formulation:
\begin{equation}
\begin{aligned}
\label{eq.ndyn}
\dot x(t)	&= x^{1}(t)\\
\dot x^{1}(t)	&= x^{2}(t)\\
			&\dotso\\
\dot x^{n-1}(t) &= V(x(t),x^{1}(t),\dotsc,x^{n-1}(t)).
\end{aligned}
\end{equation}
An $n$-th order dynamical system on $\strat$ will thus correspond to a flow on the \emph{phase space} $\phase = \coprod_{x} (\cone_{x}\strat)^{n-1}$ whose points ($n$-tuples of the form $(x,x^{1},\dotsc,x^{n-1})$ as above) represent all possible \emph{states} of the system;%
\footnote{By convention, we let $\left(\cone_{x}\strat\right)^{0} = \{0\}$, so $\phase = \coprod_{x} \{0\} \cong \strat$ for $n=1$.}
%More concisely, a state may be viewed as an \emph{$(n-1)$-jet}, with the phase space $\phase$ corresponding to the \emph{$(n-1)$-th order jet bundle} of $\strat$ \citep{Saunders89}.}
by contrast, we will keep the designation ``points'' for base points $x\in \strat$, and $\strat$ itself will be called the \emph{configuration space} of the system.
Obviously, the evolution of an $n$-th order dynamical system depends on the entire \emph{initial state} $\omega=(x(0), \dot x(0),\dotsc, x^{(n-1)}(0))\in\phase$ and not only on $x(0)$, so stationarity and stability definitions will be phrased in terms of states $\omega\in\phase$.
On the other hand, if we wish to characterize the evolution of an initial \emph{position} $x(0)\in\strat$ over time, we will do so by means of the corresponding \emph{rest state} $(x(0), 0,\dotsc,0)$ which signifies that the system starts \emph{at rest}:
using the natural embedding $x\mapsto(x,0\dotsc,0) \in\phase$, we may thus view $\strat$ as a subset of $\phase$,
and when there is no danger of confusion, we will identify $x\in\strat$ with the associated rest state $(x,0,\dotsc,0)\in\phase$.

%----------------------------------------------------------------------
%%% DYNAMICS
%----------------------------------------------------------------------
\section{Derivation of higher order dynamics.}
\label{sec.dynamics}

A fundamental requirement for any class of game dynamics is that solution trajectories must remain in the game's strategy space $\strat$ for all time.
For a first order system of the form
\begin{equation}
\label{eq.1dyn}
\dot x_{k\alpha} = F_{k\alpha}(x),
\end{equation}
with Lipschitz $F_{k\alpha}\from\R^{\act}\to\R$, this is guaranteed by the tangency requirements
\begin{inparaenum}[\itshape a\upshape)]
\item $\insum_{\alpha}^{k} F_{k\alpha} = 0$ for all $k\in\play$, and
\item $F_{k\alpha} \geq 0 $ whenever $x_{k\alpha} = 0$.
\end{inparaenum}
In second order however, this does not suffice:
if we simply replace $\dot x_{k\alpha}$ with $\ddot x_{k\alpha}$ in \eqref{eq.1dyn} and players start with sufficiently high velocity $\dot x(0)$ pointing towards the exterior of $\strat$, then they will escape $\strat$ in finite time.

\cite{FM04} forced solutions to remain in $\strat$ by exogenously projecting the velocity $v(t) \equiv \dot x(t)$ of an orbit to the tangent cone $\cone_{x}\strat$ of ``admissible'' velocity vectors.
%\textendash\ similarly to \cite{NZ97} and \cite{SDL08}.
This approach however has the problem that projections do not vary continuously with $x$, so existence and (especially) uniqueness of solutions might fail;
moreover, players need to know exactly when they hit a boundary face of $\strat$ in order to change their projection operator, so machine precision errors are bound to arise \citep{Can00}.
To circumvent these problems, we will take an approach rooted in reinforcement learning, allowing us to respect the restrictions imposed by the simplicial structure of $\strat$ in a natural way.

\subsection{The second order replicator dynamics in dyadic games.}
\label{sec.2RD.2st}

We will first describe our higher order reinforcement learning approach in the simpler context of two-strategy games, the main idea being that players keep and update an \emph{unconstrained} measure of their strategies' payoff differences instead of updating their (constrained) strategies directly.
In this way, second (or higher) order effects arise naturally when players look two (or more) steps into the past, and it is the dynamics of these ``scores'' that induce a well-behaved dynamical system on the game's strategy space.

More precisely, consider an $n$-person game where every player $k\in\play$ has two possible actions, ``$0$'' and ``$1$'', that are appraised based on the associated payoff differences $\Delta u_{k} \equiv u_{k,1} - u_{k,0}$, $k\in\play$.
With this regret-like information at hand, players can measure the performance of their strategies over time by updating the auxiliary \emph{score variables} (or \emph{propensities}):
\begin{equation}
\label{eq.score.2st}
U_{k}(t+h) = U_{k}(t) + h \Delta u_{k}(x(t)),
\end{equation}
where $h$ is the time interval between updates and $\Delta u_{k}(x(t))$ represents the payoff difference between actions ``1'' and ``0'' at time $t$ (assumed discrete for the moment).
The players' strategies $x_{k}\in\strat_{k}$ are then updated following the \emph{logit} (or \emph{exponential weight}) choice model whereby actions that score higher are played exponentially more often \citep{Rus99,HSV09,MM10,Sor09}:
%\citep{Rus99,HSV09,MM10,Sor09}:
%\footnote{We are using the term ``inverse'' because the expit function in \eqref{eq.2logit} is actually the \emph{inverse} logit function: $U_{k} = \log (x_{k}) - \log(1-x_{k}) = \logit(x_{k})$, so $x_{k} = \logit^{-1}(U_{k}) = \expit(U_{k})$.}
\begin{equation}
\label{eq.2logit}
x_{k}(t) = \frac{\exp(U_{k}(t))}{1+\exp(U_{k}(t))}.
\end{equation}

This process is repeated indefinitely, so, for simplicity, we will descend to continuous time by letting $h\to0$ in \eqref{eq.2score.2st}.%
\footnote{We should stress that the passage to continuous time is done here at a heuristic level
\textendash\ see \cite{Rus99} and \cite{Sor09} for some of the discretization issues that arise.
This discretization is a very important topic in itself, but since our focus is the properties of the underlying continuous-time dynamics, we will not address it here.}
In this way, by collecting terms in the LHS of \eqref{eq.score.2st} and replacing the discrete difference ratio $(U_{k}(t+h) - U_{k}(t))/h$ with $\dot U_{k}$, the system of \eqref{eq.score.2st} and \eqref{eq.2logit} becomes:
\begin{subequations}
\label{eq.1XL}
\begin{flalign}
\label{eq.score.2st.cont}
\dot U_{k}
	&= \Delta u_{k}(x)\\
\label{eq.logit.2st.cont}
x_{k}
	&= \left(1+\exp(-U_{k})\right)^{-1}.
\end{flalign}
\end{subequations}
Hence, by differentiating \eqref{eq.logit.2st.cont} to decouple it from \eqref{eq.score.2st.cont}, we readily obtain the $2$-strategy replicator dynamics of \cite{TJ78}:
\begin{equation}
\dot x_{k}
	= \frac{dx_{k}}{dU_{k}} \dot U_{k}
	= x_{k} (1-x_{k})\,\Delta u_{k}(x).
\end{equation}

In this well-known derivation of the replicator dynamics from the exponential re\-in\-for\-cement rule \eqref{eq.1XL} (see also \citealp{HSV09} and \citealp{MM10}), the constraints $x_{k}\in(0,1)$, $k\in\play$, are automatically satisfied thanks to \eqref{eq.2logit}.
On the downside however, \eqref{eq.score.2st.cont} itself ``forgets'' a lot of past (and potentially useful) information because the ``discrete-time'' recursion \eqref{eq.score.2st} only looks one iteration in the past.
To remedy this, players could take \eqref{eq.score.2st} one step further by aggregating the scores $U_{k}$ themselves so as to gather more momentum towards the strategies that tend to perform better.

This reasoning yields the double-aggregation re\-in\-for\-cement scheme:
\begin{subequations}
\label{eq.2score.2st}
\begin{flalign}
U_{k}(t+h)		&= U_{k}(t) + h \Delta u_{k}(x(t))\\
Z_{k}(t+h)		&= Z_{k}(t) + h U_{k}(t),
\end{flalign}
\end{subequations}
where, as before, the profile $x(t)$ is updated following the logistic distribution \eqref{eq.2logit} applied to the double aggregate $Z$, viz. $x_{k}(t) = 1/(1 + \exp(-Z_{k}(t)))$.%
\footnote{Of course, players could look even deeper into the past by taking further aggregates in \eqref{eq.2score.2st}, but we will not deal with this issue here in order to keep our example as simple as possible.}
Thus, by eliminating the intermediate (first order) aggregation variables $U_{k}$ from \eqref{eq.2score.2st}, we obtain the \emph{second order} recursion:
\begin{equation}
\frac{Z_{k}(t+2h) - Z_{k}(t+h)}{h} = \frac{Z_{k}(t+h) - Z_{k}(t)}{h} + h \Delta u_{k}(x(t)),
\end{equation}
which in turn leads to the continuous-time variant:
%\begin{equation}
%\label{eq.2score.2st.cont}
%\ddot Z_{k} = \Delta u_{k}.
%\end{equation}
\begin{subequations}
\label{eq.2XL}
\begin{flalign}
\ddot Z_{k}
	&= \Delta u_{k}(x)\\
\label{eq.2logit.2st.cont}
x_{k}	
	&= \left(1+\exp(-Z_{k})\right)^{-1}.
\end{flalign}
\end{subequations}

The second order system \eqref{eq.2XL} automatically respects the simplicial structure of $\strat$ by virtue of the logistic updating rule \eqref{eq.2logit}, so this overcomes the hurdle of staying in $\strat$;
still, it is quite instructive to also derive the dynamics of the strategy profile $x(t)$ itself.
To that end, \eqref{eq.2logit.2st.cont} gives $Z_{k} = \log (x_{k}) - \log(1-x_{k})$, so a simple differentiation yields:
\begin{equation}
\label{eq.xzdep}
\dot Z_{k}
= \frac{\dot x_{k}}{x_{k}} + \frac{\dot x_{k}}{1-x_{k}}
= \frac{\dot x_{k}}{x_{k}(1-x_{k})}.
\end{equation}
Differentiating yet again, we thus obtain
\begin{equation}
\ddot Z_{k} = \frac{\ddot x_{k} x_{k} (1-x_{k}) - \dot x_{k}^{2} (1-x_{k}) + \dot x_{k}^{2} x_{k}}{x_{k}^{2}(1-x_{k})^{2}},
\end{equation}
and some algebra readily yields the \emph{second order replicator dynamics for dyadic games}:
\begin{equation}
\label{eq.2RD.2st}
\ddot x_{k} = x_{k} (1-x_{k}) \Delta u_{k} + \frac{1-2x_{k}}{x_{k}(1-x_{k})} \dot x_{k}^{2}.
%\hspace{30pt}
%(0<x_{k}<1).
\end{equation}

This derivation of a second order dynamical system on $\strat$ will be the archetype for the significantly more general class of higher order dynamics of the next section, so we will pause here for some remarks:

\begin{remark}[Initial Conditions]
In the first order exponential learning scheme \eqref{eq.1XL}, the players' initial scores $U_{k}(0)$ determine their initial strategies via the logit rule \eqref{eq.logit.2st.cont}, namely $x_{k}(0) = (1+\exp(-U_{k}(0)))^{-1}$.
The situation however is less straightforward in \eqref{eq.2XL} where we have two different types of initial conditions:
the players' initial scores $Z_{k}(0)$ and the associated initial velocities $\dot Z_{k}(0)$ (themselves corresponding to the initial values of the intermediate aggregation variables $U_{k}$ in \eqref{eq.2score.2st}).

In second order, the initial scores $Z_{k}(0)$ determine the players' initial strategies via \eqref{eq.2logit.2st.cont}.
On the other hand, the scores' initial velocities $\dot Z_{k}(0) = U_{k}(0)$ play a somewhat more convoluted role:
indeed, differentiating \eqref{eq.2logit.2st.cont} yields
\begin{equation}
\dot x_{k}(0)
	= \dot Z_{k}(0) \frac{\exp(Z_{k}(0))}{1+\exp(Z_{k}(0))} \frac{1- \exp(Z_{k}(0))}{1+\exp(Z_{k}(0))}
	= x_{k}(0) (1-x_{k}(0)) \dot Z_{k}(0),
\end{equation}
so the initial score velocities $\dot Z_{k}(0)$ control the initial growth rate $\dot x_{k}(0)$ of the players' strategies.
As we shall see in the following, nonzero $\dot Z_{k}(0)$ introduce an inherent bias in a players' learning scheme;
hence, given that $\dot x_{k}(0)=0$ when $\dot Z_{k}(0) = 0$ (independently of the players' initial strategy), starting ``at rest'' ($\dot x_{k}(0) = 0$) is basically equivalent to learning with no initial bias ($\dot Z_{k}(0) = 0$).
\end{remark}

\begin{remark}[Past information]
The precise sense in which the double aggregation scheme \eqref{eq.2XL} is ``looking deeper into the past'' can be understood more clearly by writing out explicitly the first and second order scores $U_{k}$ and $Z_{k}$ as follows:
\begin{subequations}
\label{eq.score.int.2st}
\begin{flalign}
\label{eq.1score.int}
U_{k}(t)	&= U_{k}(0)
		+ \int_{0}^{t}\Delta u_{k}(x(s)) \id s,\\
\label{eq.2score.int}
Z_{k}(t)	&= Z_{k}(0)
		+ \int_{0}^{t} U_{k}(s) \id s
		= Z_{k}(0)
		+ \dot Z_{k}(0) t
		+ \int_{0}^{t} (t-s)\,\Delta u_{k}(x(s)) \id s.
\end{flalign}
\end{subequations}

We thus see that the first order aggregate scores $U_{k}$ assign uniform weight to all past instances of play, while the second order aggregates $Z_{k}$ put (linearly) more weight on instances that are farther removed into the past.
This mode of weighing can be interpreted as players being reluctant to forget what has occurred, and this is precisely the reason that we describe the second order scheme \eqref{eq.2XL} as ``looking deeper into the past''.
From the point of view of learning, this may appear counter-intuitive because past information is usually discounted (e.g. by an exponential factor) and ultimately discarded in favor of more recent observations \citep{FL98}.
As we shall see, ``refreshing'' observations in this way results in the players' scores $U_{k}$ growing at most linearly in time (see e.g. \citealp{HSV09}, \citealp{Rus99}, and \citealp{Sor09});
on the flip side, if players \emph{reinforce} past observations by using \eqref{eq.2score.int} in place of \eqref{eq.1score.int}, then their scores may grow quadratically instead of linearly.
%Accordingly, first order learning tends to be more conservative (leaning towards ``exploring'' more than ``exploiting''), whereas higher order learning is less tempered and corresponds to more decisive players.

From \eqref{eq.score.int.2st} we also see that nonzero initial score velocities $\dot Z_{k}(0) \neq 0$ introduce a skew in the players' learning scheme:
for instance, in a constant game ($\Delta u_{k} \equiv 0$), the integral expression \eqref{eq.2score.int} gives $\lim_{t\to\infty} Z_{k}(t) = \sign(Z_{k}(0)) \cdot \infty$, i.e. $x_{k}(t)$ will converge to $0$ or $1$, depending only on the sign of $\dot Z_{k}(0)$.
Put differently, the bias introduced by the initial velocities $\dot Z_{k}(0)$ is not static (as the players' choice of initial strategy/score), but instead drives the player to a particular direction, even in the absence of external stimuli.

\end{remark}

\subsection{Reinforcement learning and higher order dynamics.}
\label{sec.XL}

In the general case, we will consider the following reinforcement learning setup:
\begin{enumerate}
\setlength{\itemsep}{1pt}
\setlength{\parskip}{1pt}
\item
For every action $\alpha\in\act_{k}$, player $k$ keeps and updates a \emph{score} (or \emph{propensity}) variable $y_{k\alpha}\in\R$ which measures the performance of $\alpha$ over time.

\item
Players transform the scores $y_{k}\in\R^{\act_{k}}$ into mixed strategies $x_{k}\in\strat_{k}$ by means of the Gibbs map $G_{k}\from\R^{\act_{k}}\to \strat_{k}$, $y_{k}\mapsto G_{k}(y_{k})$:
\begin{equation}
\label{eq.Gibbs}
\tag{GM}
x_{k\alpha}
= G_{k\alpha}(y_{k})
\equiv\frac
{\exp\left(\lambda_{k} y_{k\alpha}(t)\right)}
{\sum_{\beta}^{k} \exp\left(\lambda_{k} y_{k\beta}(t)\right)},
\end{equation}
where the ``inverse temperature'' $\lambda_{k}>0$ controls the model's sensitivity to external stimuli \citep{LL76}.

\item
The game is played and players record the payoffs $u_{k\alpha}(x)$ for each $\alpha\in\act_{k}$.

\item
Players update their scores and the process is repeated ad infinitum.
\end{enumerate}

Needless to say, the focal point of this learning process is the exact way in which players update the performance scores $y_{k\alpha}\in\R$ at each iteration of the game.
In the previous section, these scores were essentially defined as double aggregates of the received payoffs via the two-step process \eqref{eq.2score.2st}.
Here, we will further extend this framework by considering an $n$-fold aggregation scheme in which the scores $y_{k\alpha}$ are formed via the following $n$-step process:
\begin{equation}
\label{eq.nLD.nsteps}
\begin{aligned}
Y_{k\alpha}^{(n-1)}(t+h)	&= Y_{k\alpha}^{(n-1)}(t) + h u_{k\alpha}(x(t)),\\
Y_{k\alpha}^{(n-2)}(t+h)	&= Y_{k\alpha}^{(n-2)}(t) + hY_{k\alpha}^{(n-1)}(t)\\
					&\dotso\\
Y_{k\alpha}^{(1)}(t+h)	&= Y_{k\alpha}^{(1)}(t) + h Y_{k\alpha}^{(2)}(t),\\
y_{k\alpha}(t+h)		&= y_{k\alpha}(t) + h Y_{k\alpha}^{(1)}(t).
\end{aligned}
\end{equation}
In other words, at each update period, players first aggregate their payoffs by updating the first order aggregation variables $Y_{k\alpha}^{(n-1)}$;
they then re-aggregate these intermediate variables by updating the second order aggregation scores $Y_{k\alpha}^{(n-2)}$ above, and repeat this step up to $n$ levels, leading to the $n$-fold aggregation score $y_{k\alpha}$.
%finally, they use this $n$-th order cumulative aggregate to update their mixed strategies via the Gibbs model \eqref{eq.Gibbs}.

Similarly to the analysis of the previous section, if we eliminate the intermediate aggregation variables $Y^{(1)}$, $Y^{(2)}$ and so forth, we readily obtain the straightforward $n$-th order recursion:
\begin{equation}
\label{eq.nLD.discrete}
\frac{\Delta_{h}^{(n)} y_{k\alpha}(t)}{h^{n}} = u_{k\alpha}(x(t)),
\end{equation}
where $\Delta_{h}^{(n)} y_{k\alpha}$ denotes the $n$-th order finite difference of $y_{k\alpha}$, defined inductively as $\Delta_{h}^{(n)} y(t) = \Delta_{h}^{(n-1)} y(t+h) - \Delta_{h}^{(n-1)} y(t)$, with $\Delta_{h}^{(1)} y(t) = y(t+h) - y(t)$.%
\footnote{In fact, we have $Y_{k\alpha}^{(r)}(t) = \Delta_{h}^{(r)} y_{k\alpha}(t) \big/ h^{r}$ for all $r=1,\dotsc, n-1$, explaining our choice of notation.}
Thus, if we descend to continuous time by letting $h\to0$, we obtain the \emph{$n$-th order learning dynamics}:
\begin{equation}
\label{eq.nLD}
\tag{LD$_{n}$}
y_{k\alpha}^{(n)}(t) = u_{k\alpha}(x(t)),
\end{equation}
with $x(t)$ given by the Gibbs map \eqref{eq.Gibbs} applied to $y(t)$.

\smallskip

The learning dynamics \eqref{eq.nLD} together with the logit choice model \eqref{eq.Gibbs} completely specify the evolution of the players' mixed strategy profile $x(t)$ and will thus constitute the core of our considerations.
However, it will also be important to derive the associated higher order dynamics induced by \eqref{eq.nLD} on the players' strategy space $\strat$;
to that end, we begin with the identity
\begin{equation}
\label{eq.logdiff}
\log(x_{k\alpha}) - \log(x_{k\beta})
= \lambda_{k}(y_{k\alpha} - y_{k\beta}),
\end{equation}
itself an easy consequence of \eqref{eq.Gibbs}.
By Fa\`a di Bruno's higher order chain rule \citep{Fra78}, we then obtain
\begin{equation}
\label{eq.nchain}
\ddt[n] \log(x_{k\alpha}(t))
= \sum\frac{n!}{m_{1}!\dotsm m_{n}!} \frac{(-1)^{m-1}(m-1)!}{x_{k\alpha}^{m}}
\inprod_{r=1}^{n} \left(x_{k\alpha}^{(r)}(t)\big/r!\right)^{m_{r}},
\end{equation}
where $m = m_{1}+\dotsm+m_{n}$, and the sum is taken over all non-negative integers $m_{1},\dotsc,m_{n}$ such that $\sum_{r=1}^{n} r m_{r} = n$.
In particular, since the only term that contains $x_{k\alpha}^{(n)}$ has $m_{1} = m_{2} = \dotso = m_{n-1} = 0$ and $m_{n}=1$, we may rewrite \eqref{eq.nchain} as
\begin{equation}
\label{eq.nderx}
\ddt[n] \log(x_{k\alpha}(t))
= \frac{x_{k\alpha}^{(n)}(t)}{x_{k\alpha}(t)}
+ R_{k\alpha}^{n-1}\big(x(t),\dot x(t),\dotsc, x^{(n-1)}(t)\big),
\end{equation}
where $R_{k\alpha}^{n-1}$ denotes the $(n-1)$-th order remainder of the RHS of \eqref{eq.nchain}:
\begin{equation}
\label{eq.adjustment}
R_{k\alpha}^{n-1}\left(x,\dot x,\dotsc,x^{(n-1)}\right)
= \sum\frac{(-1)^{m-1} n!}{m_{1}!\dotsm m_{n-1}!} \frac{(m-1)!}{x_{k\alpha}^{m}}
\prod_{r=1}^{n-1} \left(x_{k\alpha}^{(r)}(t)\big/r!\right)^{m_{r}}.
\end{equation}

By taking the $n$-th derivative of \eqref{eq.logdiff} and substituting, we thus get
\begin{equation}
\label{eq.pre.nRD}
\lambda_{k}\left(u_{k\alpha} - u_{k\beta}\right)
= \frac{x_{k\alpha}^{(n)}}{x_{k\alpha}} - \frac{x_{k\beta}^{(n)}}{x_{k\beta}}
+ R_{k\alpha}^{n-1} - R_{k\beta}^{n-1},
\end{equation}
so, by multiplying both sides with $x_{k\beta}$ and summing over $\beta\in\act_{k}$ (recall that $\sum_{\beta}^{k} x_{k\beta}^{(n)}=0$), we finally obtain the \emph{$n$-th order (asymmetric) replicator dynamics}:
\begin{equation}
\label{eq.nRD}
\tag{RD$_{n}$}
x_{k\alpha}^{(n)}
= \lambda_{k} x_{k\alpha} \left(u_{k\alpha} - u_{k}\right)
- x_{k\alpha} \left(R_{k\alpha}^{n-1} - \insum_{\beta}^{k} x_{k\beta} R_{k\beta}^{n-1}\right).
\end{equation}

\smallskip

The higher order replicator equation \eqref{eq.nRD} above will be the chief focus of our paper;
as such, a few remarks are in order:

\setcounter{remark}{0}

\begin{remark}
As one would expect, for $n=1$, we trivially obtain $R^{0}_{k\alpha} = 0$ for all $\alpha\in\act_{k}$, $k\in\play$, so \eqref{eq.nRD} reduces to the standard (asymmetric) replicator dynamics of \cite{TJ78}:
\begin{equation}
\label{eq.1RD}
\tag{RD$_{1}$}
\dot x_{k\alpha}
= \lambda_{k} x_{k\alpha} \left(u_{k\alpha}(x) - u_{k}(x)\right).
\end{equation}
On the other hand, for $n=2$, the only lower order term that survives in \eqref{eq.adjustment} is for $m_{1}=2$;
a bit of algebra then yields the second order replicator equation:
\begin{equation}
\label{eq.2RD}
\tag{RD$_{2}$}
\txs
\ddot x_{k\alpha}
= \lambda_{k} x_{k\alpha} \left(u_{k\alpha}(x) - u_{k}(x)\right)
+ x_{k\alpha} \left(\dot x_{k\alpha}^{2} \big/ x_{k\alpha}^{2} - \insum^{k}_{\beta} \dot x_{k\beta}^{2} \big/ x_{k\beta}\right).
\end{equation}
At first glance, the above equation seems different from the dynamics \eqref{eq.2RD.2st} that we derived in Section \ref{sec.2RD.2st}, but this is just a matter of reordering: if we restrict \eqref{eq.2RD} to two strategies, ``0'' and ``1'', and set $x_{k} \equiv x_{k,1} = 1 - x_{k,0}$, we will have $\dot x_{k} = \dot x_{k,1} = -\dot x_{k,0}$, and \eqref{eq.2RD.2st} follows immediately.
%All told, the second order updating scheme $\ddot y_{k\alpha} = u_{k\alpha}$ that gives rise to \eqref{eq.2RD} highlights a very deep analogy between Newtonian mechanics and learning in games: in \eqref{eq.2RD}, the game's payoffs can be interpreted as the actual physical forces that determine the system's evolution, exactly as they would prescribe the orbits of a real physical system.
\end{remark}

\begin{remark}
In terms of structure, \eqref{eq.nRD} consists of a replicator-like term (driven by the game's payoffs) and a game-independent adjustment $R_{k\alpha}^{n-1}$ which reflects the higher order character of \eqref{eq.nRD}.
As noted by the associate editor (whom we thank for this remark), if we put the order of the dynamics aside, there is a structural similarity between the higher order replicator dynamics \eqref{eq.nRD}, the replicator dynamics with aggregate shocks of \cite{FH92}, and the stochastic replicator dynamics of exponential learning \citep{MM10}.
The reason for this similarity is that all these models are first defined in terms of an auxiliary set of variables:
absolute population sizes in \cite{FH92} and payoff scores here and in \cite{MM10}.
Differentiation of these variables with respect to time then always yields a replicator-like term carrying the game's payoffs, plus a correction term which is independent of the game being played (because it is coming from It\^o calculus or higher order considerations).
\end{remark}

\begin{remark}
\label{rem.invariance}
Technically, given that the higher order adjustment terms $R_{k\alpha}^{n-1}$ blow up for $x_{k\alpha}=0$ and $n>1$, the dynamics \eqref{eq.nRD} are only defined for strategies that lie in the (relative) interior $\mathrm{rel\,int\,}(\strat)$ of $\strat$.
If the players' initial strategy profile is itself interior, then this poses no problem to \eqref{eq.nRD} because the Gibbs map \eqref{eq.Gibbs} ensures that every strategy share $x_{k\alpha}(t)$ will remain positive for all time.
For the most part, we will not need to consider non-interior orbits;
nonetheless, if required, we can consider initial conditions on any subface $\strat'$ of $\strat$ simply by restricting \eqref{eq.Gibbs} to the corresponding restriction $\game'$ of $\game$, i.e. by effectively setting the score of an action that is not present in the initial strategy distribution to $-\infty$.
In this manner,
%by restricting the Gibbs map \eqref{eq.Gibbs} and the learning dynamics \eqref{eq.nLD} to the support of $x(0)$,
we may extend \eqref{eq.nRD} to any subface of $\strat$, and it is in this sense that we will interpret \eqref{eq.nRD} for non-interior initial conditions.
\end{remark}

\begin{remark}
In the same way that we derived the integral expressions \eqref{eq.score.int.2st} for the payoff scores, we obtain the following integral representation for the higher order learning dynamics \eqref{eq.nLD}:
\begin{equation}
\label{eq.nscore}
y_{k\alpha}(t)
	= \frac{1}{(n-1)!} \int_{0}^{t} (t-s)^{n-1} u_{k\alpha}(x(s)) \dd s
	+ \sum_{r=0}^{n-1} y_{k\alpha}^{(r)}(0) \frac{t^{r}}{r!}.
\end{equation}
As with our previous discussion in Section \ref{sec.2RD.2st}, the players' initial scores $y_{k\alpha}(0)$ determine the players' initial strategies.
Similarly, the higher order initial conditions $y_{k\alpha}^{(r)}\neq0$, $r\geq 1$, control the initial derivates $\dot x_{k\alpha}(0),\dotsc$ of \eqref{eq.nRD}, and it is easy to see that starting ``at rest'' is equivalent to having no initial learning bias that could lead players to a particular strategy in the absence of external stimuli (e.g. in a constant game; see also the concluding remarks of Section \ref{sec.2RD.2st}).
%
%would skew the players' learning scheme towards one direction or another, so it might lead players to a particular strategy even in the absence of external stimuli (e.g. in a constant game).
%For this reason
%
\end{remark}

\subsection{Evolutionary interpretations of the higher order replicator dynamics.}
\label{sec.evolution}

In the mass-action interpretation of evolutionary game theory, it is assumed that there is a nonatomic population linked to each player role $k\in \play$, and that the governing dynamics arise from individual interactions within these populations.
%Unsurprisingly, the dynamics \eqref{eq.nRD} admit a similar interpretation by assuming that the learning process takes place within large populations of myopic agents that subscribe to a common learning scheme.
%
%More precisely, let $x_{k\alpha}(t)$ denote the population share of $\alpha$-type strategists in the $k$-th player population, and assume that individuals monitor the (now global) performance scores $y_{k\alpha}$ associated to their population, themselves following the higher order aggregation scheme \eqref{eq.nLD}.%
%\footnote{In this interpretation, the second (or higher) order evolution of $y$ could represent time lags and delayed population effects as in \cite{MS71} and \cite{LSP77}.}
%Then, if each (nonatomic) player chooses a strategy at time $t$ with probability proportional to $\exp(y_{k\alpha}(t))$, our previous discussion shows that population evolution will follow the $n$-th order replicator dynamics \eqref{eq.nRD}.
%
In the context of \eqref{eq.nRD}, such an evolutionary interpretation may be obtained as follows:
focusing on the case $n=2$ for simplicity, assume that each (nonatomic) player receives an opportunity to switch strategies at every ring of a Poisson alarm clock as described in detail in Chapter 3 of \cite{San10}.
In this context, if $\rho_{\alpha\beta}^{k}$ denotes the \emph{conditional switch rate} from strategy $\alpha\in\act_{k}$ to strategy $\beta\in\act_{k}$ in population $k\in\play$ (i.e. the probability of an $\alpha$-strategist becoming a $\beta$-strategist up to a normalization factor), then the strategy shares $x_{k\alpha}$ will follow the \emph{mean dynamics associated to $\rho$}:
\begin{equation}
\label{eq.MD}
\tag{MD$_{\rho}$}
\dot x_{k\alpha} = \insum_{\beta}^{k} x_{k\beta} \rho_{\beta\alpha}^{k} - x_{k\alpha} \insum_{\beta}^{k} \rho_{\alpha\beta}^{k}.
\end{equation}

Conditional switch rates are usually functions of the current population state $x\in\strat$ and the corresponding payoffs $u_{k\alpha}$:
for instance, the well-known ``imitation of success'' revision protocol is described by the rule
\begin{equation}
\label{eq.success1}
\rho_{\alpha\beta}^{k} = x_{k\beta} u_{k\beta}(x),
\end{equation}
and the resulting mean field \eqref{eq.MD} is simply the standard replicator equation \citep{Hof95,San10}.%
\footnote{Other revision protocols that lead to the replicator dynamics are Schlag's (1998)``pairwise proportional imitation'' and the protocol of ``pure imitation driven by dissatisfaction'' of \cite{BW96}.
We will only focus here on ``imitation of success'' for simplicity; that said, the discussion that follows can easily be adapted to these revision protocols as well.}
On the other hand, if players are more sophisticated and keep track of the \emph{long-term} performance $U_{k\alpha}(t) = \int_{0}^{t} u_{k\alpha}(x(s)) \dd s$ of a strategy over time (instead of only considering the instantaneous payoffs $u_{k\alpha}$), then the ``long-term'' analogue of the revision rule \eqref{eq.success1} will be
\begin{equation}
\label{eq.success2}
\tilde\rho_{\alpha\beta}^{k}
	= x_{k\beta} U_{k\beta}
	= x_{k\beta} \int_{0}^{t} u_{k\beta}(x(s)) \dd s,
\end{equation}
leading in turn to the mean dynamics:
\begin{equation}
\label{eq.MD2}
\dot x_{k\alpha} = x_{k\alpha} \left(U_{k\alpha} - \insum_{\beta}^{k} x_{k\beta} U_{k\beta}\right).
\end{equation}

Of course, \eqref{eq.MD2} is not a dynamical system per se, but a system of integro-differential equations (recall that $U_{k\alpha}$ has an integral dependence on $x$).
However, by differentiating \eqref{eq.MD2} with respect to time and recalling that $\dot U_{k\alpha} = u_{k\alpha}$, we readily obtain:
\begin{equation}
\label{eq.MD2.der}
\ddot x_{k\alpha}
	= \dot x_{k\alpha} \left(U_{k\alpha} - \insum_{\beta}^{k} x_{k\beta} U_{k\beta}\right)
	+ x_{k\alpha}\left(u_{k\alpha} - \insum_{\beta}^{k} x_{k\beta} u_{k\beta}\right)
	- x_{k\alpha} \insum_{\beta}^{k} \dot x_{k\beta} U_{k\beta}.
\end{equation}
By \eqref{eq.MD2}, the first term in the RHS of \eqref{eq.MD2.der} above will be equal to $\dot x_{k\alpha}^{2}\big/x_{k\alpha}$;
moreover, some easy algebra also yields
\begin{equation}
\begin{aligned}
\insum_{\beta}^{k} \dot x_{k\beta} U_{k\beta}
	&= \insum_{\beta}^{k} x_{k\beta} \left(U_{k\beta} - \insum_{\gamma}^{k} x_{k\gamma} U_{k\gamma}\right) U_{k\beta}\\
	&= \insum_{\beta}^{k} x_{k\beta} U_{k\beta}^{2} - \left(\insum_{\beta}^{k} x_{k\beta} U_{k\beta}\right)^{2}\\
	&= \insum_{\beta}^{k} x_{k\beta} \left(U_{k\beta} - \insum_{\gamma}^{k} x_{k\gamma} U_{k\gamma}\right)^{2}
	= \insum_{\beta} \dot x_{k\beta}^{2}\big/x_{k\beta}.
\end{aligned}
\end{equation}
Thus, after some rearranging, \eqref{eq.MD2.der} becomes
\begin{equation}
\ddot x_{k\alpha}
= x_{k\alpha} \left(u_{k\alpha}(x) - u_{k}(x)\right)
+ x_{k\alpha} \left(\dot x_{k\alpha}^{2} \big/ x_{k\alpha}^{2} - \insum^{k}_{\beta} \dot x_{k\beta}^{2} \big/ x_{k\beta}\right),
\end{equation}
i.e. the mean dynamics associated to the ``imitation of long-term success'' revision protocol \eqref{eq.success2} is just the second order replicator equation \eqref{eq.2RD} with $\lambda_{k}=1$.

The higher order dynamics \eqref{eq.nRD} may be derived from similar considerations, simply by taking a revision protocol of the form \eqref{eq.success2} with $U$ replaced by a different (higher order) payoff aggregation scheme.
Accordingly, the evolutionary significance of higher order is similar to its learning interpretation:
higher order dynamics arise when players revise their strategies based on long-term performance estimates instead of instantaneous payoff information.
Obviously, this opens the door to higher order variants of other population dynamics that arise from revision protocols (such as the Smith dynamics and other pairwise comparison dynamics), but since this discussion would take us too far afield, we will delegate it to a future paper.

\begin{remark*}
The definition of the payoff aggregates $U_{k\beta}(t) = \int_{0}^{t} u_{k\beta}(x(s)) \dd s$ gives $\dot x_{k\beta}(0) = 0$ in \eqref{eq.MD2}, so players will be starting ``at rest'' in \eqref{eq.MD2.der}.
On the other hand, just as in \eqref{eq.1score.int}, if players are inherently predisposed towards one strategy or another, these aggregates could be offset by some nonzero initial bias $U_{k\beta}(0)$, which would then translate into a nonzero initial velocity $\dot x(0)$.
%revision protocol \eqref{eq.success2} implies that players start ``at rest'', irrespective of their initial strategy profile $x(0)$:
%indeed, $U_{k\beta}(0) = \int_{0}^{0} u_{k\beta}(x(s)) \dd s = 0$, so \eqref{eq.MD2.der} yields $\dot x_{k\beta}(0)$ for all $\beta\in\act_{k}$, $k\in\play$.
%indeed, the definition of the payoff aggregates $U_{k\beta}(t) = \int_{0}^{t} u_{k\beta}(x(s)) \dd s$ yields $U_{k\beta}(0) =0$, so we will also have $\dot x_{k\beta}(0) = 0$ in \eqref{eq.MD2}.
In view of the above, when it is safe to assume that players are not ex-ante biased, starting at rest will be our baseline assumption.
%the population interpretation of the second order replicator dynamics \eqref{eq.2RD} via the revision protocol \eqref{eq.success2} provides further justification for beginning play at rest, just like players who use the higher order learning scheme \eqref{eq.nLD} are unlikely to introduce an ex-ante bias to their learning process.
\end{remark*}

%\begin{remark}[Time averages]
%Importantly, instead of aggregating payoffs over time, players could keep a long-term average of their strategies' payoffs and use the averaging estimator
%\begin{equation}
%\overline u_{k\beta}(t) = \frac{1}{t} \int_{0}^{t} u_{k\beta}(x(s)) \dd s = U_{k\beta}(t)/t,
%\end{equation}
%which leads in turn to the mean dynamics
%\begin{equation}
%\dot x_{k\alpha}
%	= x_{k\alpha} \left(\overline u_{k\alpha} - \insum_{\beta}^{k} x_{k\beta} \overline u_{k\beta}\right)
%	= t^{-1} x_{k\alpha} \left(U_{k\alpha} - \insum_{\beta}^{k} x_{k\beta} U_{k\beta}\right).
%\end{equation}
%After rescaling time to $\tau = \log t$, we see that this last equation is essentially equivalent to \eqref{eq.MD2}.
%As a result, estimating a strategy's performance over time by aggregating or averaging payoffs yields the same asymptotic behavior;
%however, evolution based on averaging payoffs will be logarithmically slower than evolution driven by aggregating.
%\end{remark}

%----------------------------------------------------------------------
%%% DOMINATED STRATEGIES
%----------------------------------------------------------------------
\section{Elimination of dominated strategies.}
\label{sec.dominance}

A fundamental rationality requirement for any class of game dynamics is that dominated strategies die out in the long run.
Formally, if play evolves over time, say along the path $x(t)$, $t\geq0$, we will say that the strategy $\alpha\in\act_{k}$ \emph{becomes extinct along $x(t)$} if $x_{k\alpha}(t) \to 0$ as $t \to\infty$;
more generally, for mixed strategies $q_{k}\in\strat_{k}$, we will follow \cite{SZ92} and say that $q_{k}$ becomes extinct along $x(t)$ if $\min\{x_{k\alpha}(t):\alpha\in\supp(q_{k})\}\to0$, with the minimum taken over the support $\supp(q_{k}) \equiv \{\beta\in\act_{k}:q_{k\beta}>0\}$ of $q_{k}$.
Equivalently, if we let
\begin{equation}
\label{eq.KL}
\dkl(q_{k}\midd x_{k})
= \insum_{\alpha} q_{k\alpha} \log\left(q_{k\alpha}\big/ x_{k\alpha}\right)
\end{equation}
denote the \emph{Kullback-Leibler divergence} of $x_{k}$ with respect to $q_{k}$ (with the usual convention $0\cdot \log 0=0$ when $q_{k\alpha}=0$), then $\dkl(q_{k}\midd x_{k})$ blows up to $+\infty$ whenever $\min\{x_{k\alpha}:\alpha\in\supp(q_{k})\}\to0$, so $q_{k}\in\strat_{k}$ becomes extinct along $x(t)$ if and only if $\dkl(q_{k}\midd x_{k}(t))\to\infty$ as $t\to\infty$ (see e.g. \citealt{Wei95}).

In light of the above, our first result is to show that in the $n$-th order replicator dynamics, dominated strategies die out at a rate which is exponential in $t^{n}$:

\begin{theorem}
\label{thm.dom}
Let $x(t)$ be an interior solution orbit of the $n$-th order replicator dynamics \eqref{eq.nRD}. If $q_{k}\in\strat_{k}$ is iteratively dominated, we will have
\begin{equation}
\label{eq.domrate.mixed}
\dkl(q_{k}\midd x_{k}(t)) \geq \lambda_{k} c t^{n}\big/n! + \bigoh(t^{n-1}),
\end{equation}
for some constant $c>0$.
%in other words, only iteratively undominated strategies survive.
In particular, for pure strategies $\alpha\prec\beta$, we will have
\begin{equation}
\label{eq.domrate.pure}
x_{k\alpha}(t)\big/x_{k\beta}(t)
\leq \exp\left(-\lambda_{k} \Delta u_{\beta\alpha} t^{n}\big/n! + \bigoh(t^{n-1})\right),
\end{equation}
where $\Delta u_{\beta\alpha} = \min_{x\in\strat}\{u_{k\beta}(x) - u_{k\alpha}(x)\}>0$.
\end{theorem}

\begin{figure}[t]
\centering
\subfigure[Portrait of a dominance solvable game.]{
\label{subfig.dom.portrait}
\includegraphics[width=0.47\columnwidth]{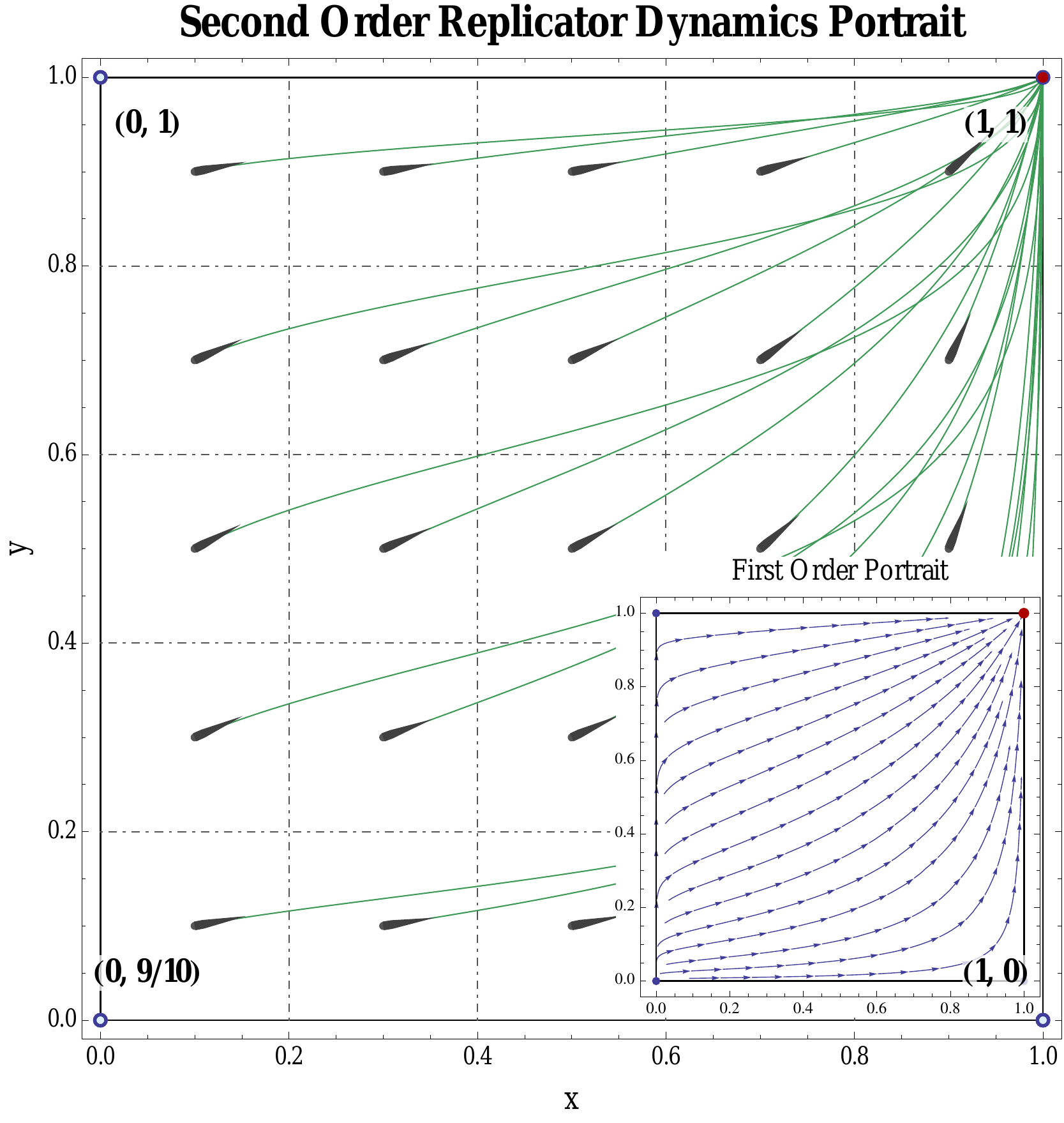}}
\hfill
\subfigure[Rate of extinction of dominated strategies.]{
\label{subfig.dom.rate}
\includegraphics[width=0.49\columnwidth]{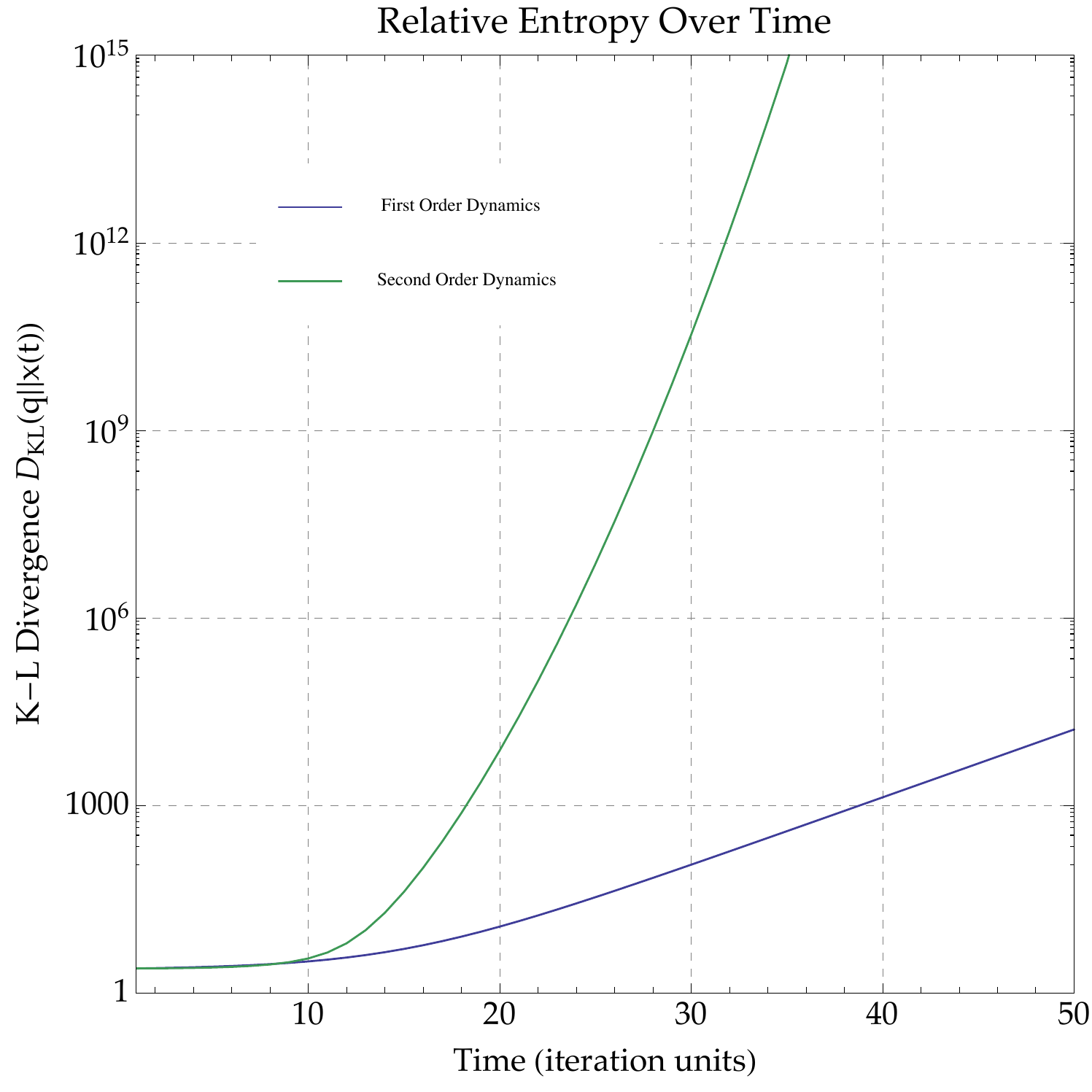}}
%\vspace{-5pt}
\caption{
\footnotesize
Extinction of dominated strategies in the first and second order replicator dynamics.
In Fig.~\ref{subfig.dom.portrait} we plot the second order solution orbits of a dominance solvable game with payoff matrices
$U_{1} = ((1,1),(0,0))$ for the ``$x$'' player and $U_{2} = ((1,0),(1,9/10))$ for the ``$y$'' player (see also the figure's labels).
In Fig.~\ref{subfig.dom.rate}, we illustrate the rate of extinction of the dominated strategy of player $1$ by plotting the corresponding K-L divergence of a typical trajectory:
the K-L distance grows exp-quadratically in second order dynamics compared to exp-linearly in first order.}
\label{fig.dom}
%\vspace{-10pt}
\end{figure}

As an immediate corollary, we then obtain:

\begin{corollary}
\label{cor.dom}
In dominance-solvable games, the $n$-th order replicator dynamics \eqref{eq.nRD} converge to the game's rational solution.
%at a rate which is at least exponential in $t^{n}$.
\end{corollary}

\setcounter{remark}{0}

\begin{remark*}
Before proving Theorem \ref{thm.dom}, it is worth nothing that even though \eqref{eq.domrate.mixed} and \eqref{eq.domrate.pure} have been stated as inequalities, one can use any upper bound for the game's payoffs to show that the rate of extinction of dominated strategies in terms of the K-L divergence is indeed $\bigoh(t^{n})$.%
\footnote{In fact, the coefficients that make \eqref{eq.domrate.mixed} and \eqref{eq.domrate.pure} into asymptotic equalities can also be determined, but we will not bother with this calculation here.}
As a result, the asymptotic rate of extinction of dominated strategies in the $n$-th order replicator dynamics \eqref{eq.nRD} is $n$ orders as fast as in the standard first order dynamics \eqref{eq.1RD},
%Thus, assuming for simplicity that players start at rest (i.e. with no initial ``velocity'' or other higher order derivate that could provide an unwarranted initial bias towards one strategy or another), we see that
so irrational play becomes extinct much faster in higher orders.
\end{remark*}

\begin{proof}[Proof of Theorem \ref{thm.dom}]
We will begin by showing that if $q_{k}\in\strat_{k}$ is dominated by $q_{k}'\in\strat_{k}$, then $\dkl(q_{k}\midd x_{k}(t)) \geq c t^{n}\big/n!$ for some positive constant $c>0$.
Indeed, let $V_{k}(x) = \dkl(q_{k}\midd x_{k}) - \dkl(q'_{k}\midd x_{k})$, and rewrite \eqref{eq.Gibbs} as $\log x_{k\alpha} = \lambda_{k} y_{k\alpha} - \log(\pf_{k}(y))$ where $\pf_{k}(y)=\sum_{\beta}^{k} \exp(\lambda_{k} y_{k\beta})$ denotes the partition function of player $k$.
Then, some algebra yields:
\begin{flalign}
V_{k}(x)	&=
	\insum_{\alpha\in\supp(q)} q_{k\alpha} \log \left(q_{k\alpha}\big/ x_{k\alpha}\right)\;\;
	- \insum_{\alpha\in\supp(q')} q_{k\alpha}' \log \left(q_{k\alpha}'\big/ x_{k\alpha}\right)\notag\\
	&= \insum_{\alpha}^{k} \big(q_{k\alpha}' - q_{k\alpha}\big) \log x_{k\alpha} + h_{k}(q_{k},q_{k}')\notag\\
	&= \insum_{\alpha}^{k} \big(q_{k\alpha}' - q_{k\alpha}\big) \lambda_{k} y_{k\alpha} + h_{k}(q_{k},q_{k}'),
\end{flalign}
where $h_{k}(q_{k},q_{k}')$ is a constant depending only on $q_{k}$ and $q_{k}'$, and the last equality follows from the fact that $\sum_{\alpha}^{k} (q_{k\alpha}' - q_{k\alpha}) \log \pf_{k} =0$ (recall that $\sum_{\alpha}^{k} q_{k\alpha} = \sum_{\alpha}^{k} q_{k\alpha}' = 1$).
In this way, we obtain:
\begin{flalign}
\label{eq.diffdom}
\ddt[n] V_{k}(x(t))
	&= \lambda_{k} \insum_{\alpha}^{k} \big(q_{k\alpha}' - q_{k\alpha}\big) y_{k\alpha}^{(n)}
	= \lambda_{k} \insum_{\alpha}^{k} \big(q_{k\alpha}' - q_{k\alpha}\big) u_{k\alpha}(x(t))\notag\\
	&=\lambda_{k} \left[u_{k}(q_{k}';x_{-k}(t)) - u_{k}(q_{k};x_{-k}(t))\right]
	\geq \lambda_{k} \Delta u_{k} >0,
\end{flalign}
where the constant $\Delta u_{k}$ is defined as $\Delta u_{k} = \min_{\strat_{-k}}\{u_{k}(q_{k}';x_{-k}) - u_{k}(q_{k};x_{-k})\}$ and its positivity follows from the fact that $\strat$ is compact and $u_{k}$ is continuous.
Hence, if we set $c_{r} = (r!)^{-1} \left.\frac{d^{r} V_{k}}{dt^{r}}\right|_{t=0}$, $r=0\dotsc n-1$, Taylor's theorem with Lagrange remainder readily gives:
\begin{equation}
\label{eq.domrate.estimate}
V_{k}(x(t)) \geq \lambda_{k} \Delta u_{k} t^{n}\big/n! + \insum_{r=0}^{n-1} c_{r} t^{r},
\end{equation}
and our assertion follows by noting that $\dkl(q_{k}\midd x_{k}(t))\geq V_{k}(x(t))$.
In particular, for pure strategies $\alpha\prec\beta$, we will have $V_{k}(x(t)) = \log x_{k\beta}(t) - \log x_{k\alpha}(t)$, so \eqref{eq.domrate.estimate} gives:
\begin{equation}
\log\big(x_{k\beta}(t)\big/ x_{k\alpha}(t)\big)
\geq \lambda_{k}\, \Delta u_{\beta\alpha} t^{n}\big/n! + \bigoh(t^{n-1}),
\end{equation}
and \eqref{eq.domrate.pure} follows by exponentiating.

Now, to establish the theorem for \emph{iteratively} dominated strategies, we will resort to induction on the rounds of elimination.
To that end, let $X_{k}^{r}$ denote the space of strategies that survive $r$ elimination rounds, and assume that $\dkl(q_{k}\midd x_{k}(t)) = \bigoh(t^{n})$ for all strategies $q_{k}\notin\strat_{k}^{r}$;
in particular, if $\alpha\notin\act_{k}^{r}\equiv\act_{k}\cap\strat_{k}^{r}$, this implies that $x_{k\alpha}(t)\to 0$ as $t\to\infty$.
We will show that this also holds if $q_{k}\in\strat_{k}^{r}$ survives for $r$ deletion rounds but dies in the subsequent one.
Indeed, if $q_{k}\in\strat_{k}^{r}\exclude\strat_{k}^{r+1}$, there will be some $q_{k}'\in\strat_{k}^{r}$ with $u_{k}(q_{k}';x_{-k}) > u_{k}(q_{k};x_{-k})$ for all $x_{-k}\in\strat_{-k}^{r}$.
With this in mind, decompose $x\in\strat$ as $x=x^{r}+z^{r}$ where $x^{r}$ denotes the ``$r$-ra\-tio\-na\-lizable'' part of $x$, i.e. the orthogonal projection of $x$ on the subspace of $\strat$ spanned by the surviving pure strategies $\act_{\ell}^{r}$, $\ell\in\play$.
Then, if we set $\Delta u_{k}^{r} = \min\{u_{k}(q_{k}';\alpha_{-k}) - u_{k}(q_{k};\alpha_{-k}):\alpha_{-k}\in\act_{-k}^{r}\}$, we will also have:
\begin{equation}
\label{eq.diff.adm}
u_{k}(q_{k}';x_{-k}^{r}) - u_{k}(q_{k};x_{-k}^{r}) \geq \Delta u_{k}^{r}>0\quad
\text{for all $x_{-k}\in\strat_{-k}$}.
\end{equation}
Moreover, it is easy to see that our induction hypothesis implies $z^{r}(t)\to 0$ as $t\to\infty$ (recall that $x_{k\alpha}(t)\to0$ for all $\alpha\notin\act_{k}^{r}$), so, for large enough $t$, we also get:
\begin{equation}
\label{eq.diff.dom}
|u_{k}(q_{k}';z_{-k}^{r}(t)) - u_{k}(q_{k};z_{-k}^{r}(t))| < \Delta u_{k}^{r}/2.
\end{equation}
Hence, by combining \eqref{eq.diff.adm} and \eqref{eq.diff.dom}, we obtain $u_{k}(q_{k}';x_{-k}(t)) - u_{k}(q_{k};x_{-k}(t)) > \Delta u_{k}^{r}/2$ for large $t$, and the induction is complete by plugging this last estimate into \eqref{eq.diffdom} and proceeding as in the base case $r=0$ (our earlier assertion).
\end{proof}

On the other hand, if a strategy is only \emph{weakly} dominated,
the payoff differences $\Delta u_{\beta\alpha}$ in \eqref{eq.domrate.pure} and related estimates vanish, so
Theorem \ref{thm.dom} cannot guarantee that it will be annihilated.
In fact, it is well-known that weakly dominated strategies may survive in the standard first order replicator dynamics:
if the pure strategy $\alpha\in\act_{k}$ of player $k$ is weakly dominated by $\beta\in\act_{k}$, and if all adversarial strategies $\alpha_{-k}\in\act_{-k}$ against which $\beta$ performs better than $\alpha$ die out, then $\alpha$ may survive for an open set of initial conditions \citep[for instance, see Example~5.4 and Proposition 5.8 in][]{Wei95}.

Quite remarkably, this can never be the case in a higher order setting if players start unbiased:

\begin{theorem}
\label{thm.dom.weak}
Let $x(t)$ be an interior solution orbit of the $n$-th order ($n\geq2$) replicator dynamics \eqref{eq.nRD} that starts at rest:
$\dot x(0) = \dotso = x^{(n-1)}(0) = 0$.
If $q_{k}\in\strat_{k}$ is weakly dominated, then it becomes extinct along $x(t)$ with rate
\begin{equation}
\label{eq.domrate.weak}
\dkl(q_{k}\midd x_{k}(t)) \geq \lambda_{k} c t^{n-1}\big/(n-1)!,
\end{equation}
where $\lambda_{k}$ is the learning rate of player $k$ and $c>0$ is a positive constant.
\end{theorem}

The intuition behind this surprising result can be gleaned by looking at the reinforcement learning scheme \eqref{eq.nLD}.
If we take the case $n=2$ for simplicity, we see that the ``payoff forces'' $F_{k\alpha} \equiv u_{k\alpha}$ never point towards a weakly dominated strategy.
As a result, solution trajectories are always accelerated away from weakly dominated strategies, and even if this acceleration vanishes in the long run, the trajectory still retains a growth rate that drives it away from the dominated strategy.
By comparison, this is not the case in first order dynamics:
there, we only know that \emph{growth rates} point away from weakly dominated strategies, and if these rates vanish in the long run, solution trajectories might ultimately converge to a point where weakly dominated strategies are still present (see for instance Fig.~\ref{fig.dom.weak}.
The proof follows by formalizing these ideas:

\begin{proof}[Proof of Theorem \ref{thm.dom.weak}]
Let $q_{k}\preccurlyeq q_{k}'$ and let $\act_{-k}' \equiv \{\alpha_{-k}\in\act_{-k}: u_{k}(q_{k}';\alpha_{-k}) > u_{k}(q_{k};\alpha_{-k})\}$ be the set of pure strategy profiles of $k$'s opponents against which $q_{k}'$ yields a strictly greater payoff than $q_{k}$. Then, with notation as in the Proof of Theorem \ref{thm.dom}, we will have:
\begin{equation}
\label{eq.diffdom.weak.diff}
\ddt[n] V_{k}(x(t))
= \lambda_{k} \insum_{\alpha_{-k}\in\act_{-k}'}
\big[u_{k}(q_{k}';\alpha_{-k}) - u_{k}(q_{k};\alpha_{-k})\big]
x_{\alpha_{-k}}(t),
\end{equation}
%where $x_{-k}^{*} = \min\{x_{\alpha_{k}}:\alpha_{-k}\in\act_{-k}^{*}\}$, and
where $x_{\alpha_{-k}} \equiv \prod_{\ell\neq k} x_{\alpha_{\ell}}$ denotes the $\alpha_{-k}$-th component of $x$. Thus, with $x(t)$ starting at rest, Fa\`a di Bruno's formula gives $\left.\frac{d^{r}V_{k}}{dt^{r}}\right|_{t=0}=0$ for all $r=1,\dotsc,n-1$, and a simple integration then yields:
\begin{equation}
\label{eq.diffdom.weak}
\ddt[n-1] V_{k}(x(t))
= \lambda_{k}
\insum_{\alpha_{-k}\in\act_{-k}'}
\big[u_{k}(q_{k}';\alpha_{-k}) - u_{k}(q_{k};\alpha_{-k})\big]
\int_{0}^{t} x_{\alpha_{-k}}(s) \id s,
\end{equation}
However, with $x(t)$ interior, the integrals in the above equation will be positive and increasing, so for some suitably chosen $c>0$ and $t$ large enough, we obtain
\begin{equation}
\label{eq.diffdom.weak.constant}
\ddt[n-1] V_{k}(x(t))
%\geq \lambda_{k}
%\int_{0}^{T} x_{-k}^{*}(s) \dd s
\geq \lambda_{k} c
> 0,
%\text{for all sufficiently large $t$,}
\end{equation}
%for some suitably chosen $c>0$,
and our claim follows from a ($n-1$)-fold integration.
\end{proof}

\begin{figure}[t]
\centering
\subfigure[Entry Deterrence]{
\label{subfig.EntryDeterrence}
\includegraphics[width=0.45\columnwidth]{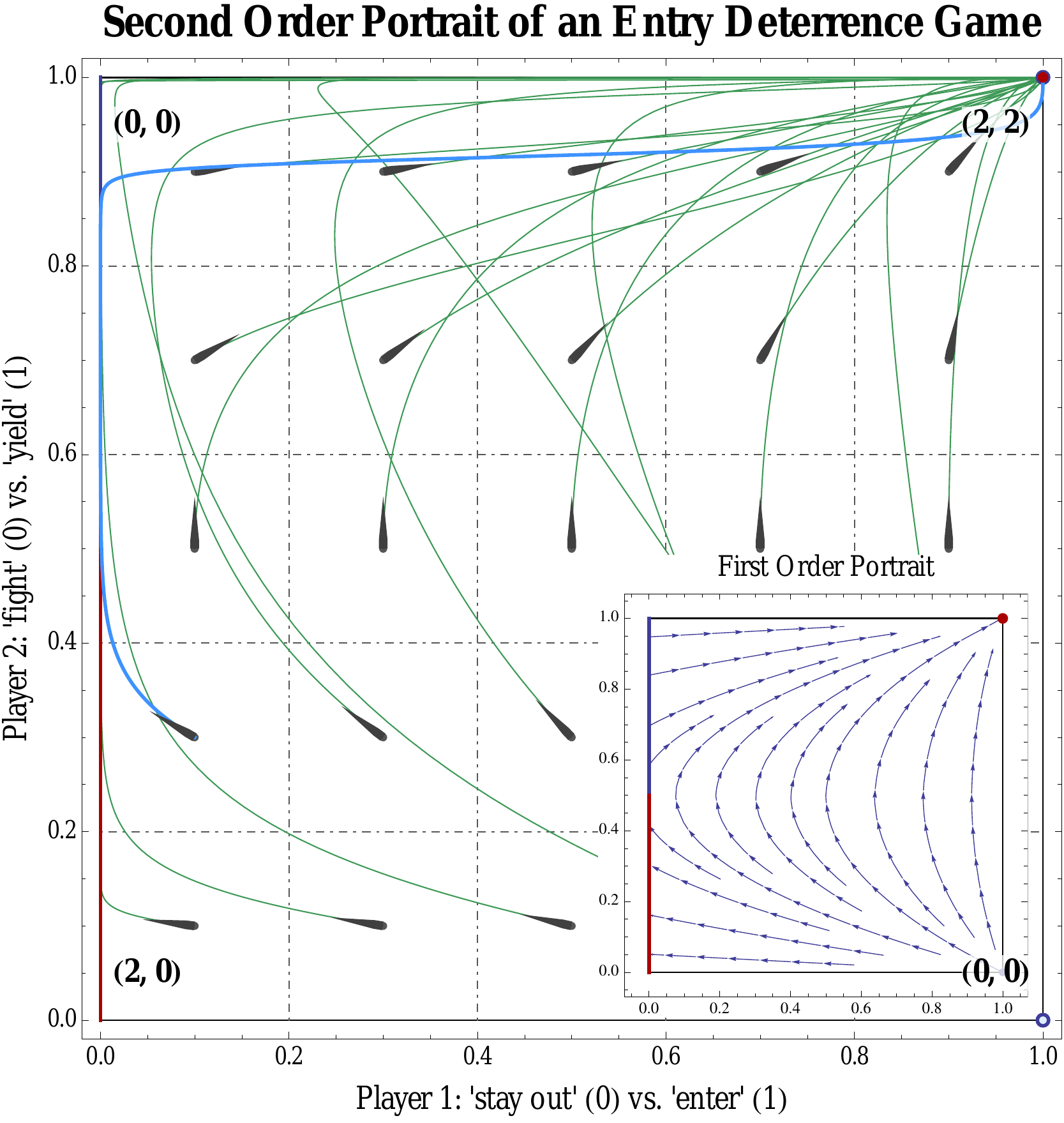}}
\hfill
\subfigure[Outside Option]{
\label{subfig.OutsideOption}
\includegraphics[width=0.47\columnwidth]{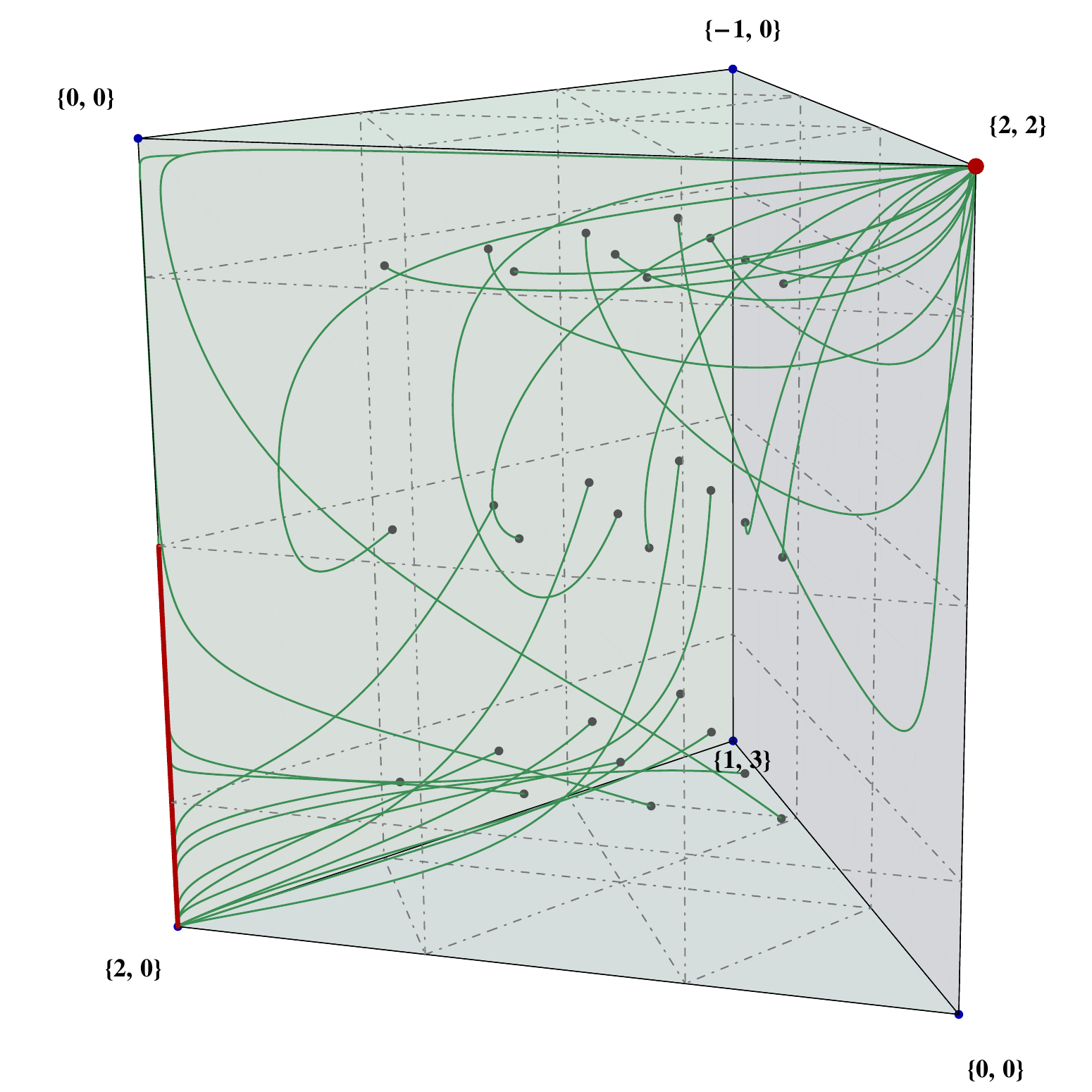}}
%\vspace{-10pt}
\caption{
\footnotesize
Extinction of weakly dominated strategies and survival of iteratively weakly dominated ones in the second order replicator dynamics.
Fig.~\ref{subfig.EntryDeterrence} shows solution orbits starting at rest in an Entry Deterrence game:
the weakly dominated strategy ``fight'' of Player 2 becomes extinct, in stark contrast to the first order case (compare the highlighted trajectory with the first order portrait in the inlay).
Fig.~\ref{subfig.OutsideOption} shows an Outside Option supergame where the strategy ``fight'' in Fig.~\ref{subfig.EntryDeterrence} is only \emph{iteratively} weakly dominated;
this strategy pays very well against certain initial conditions, so it ends up surviving when all evidence that it is (iteratively weakly) dominated vanishes.
(The payoff matrices for the Outside Option supergame are
$U_{1} = ((2,0),(0,2),(-1,1))$ for Player $1$ and $U_{2} = ((2,0),(0,0),(0,3))$ for Player $2$; see also the corresponding figure labels.)
%$U_{1}=((0,2),(2,0),(-1,1))$ and $U_{2} = ((0,2,0),(0,0,3))$; see also the figure's labels.)
%(In both figures, Nash equilibria have been highlighted in dark red.)
}
%\vspace{-10pt}
\label{fig.dom.weak}
\end{figure}

In view of this qualitative difference between first and higher order dynamics, some further remarks are in order:

\setcounter{remark}{0}

\begin{remark}
In the first order replicator dynamics, the elimination of weakly dominated strategies when evidence of their domination survives requires that all players adhere to the same dynamics (see e.g. the proof of Proposition 3.2 in \citealp{Wei95}).
To wit, consider a simple Entry Deterrence game where a competitor (Player 1) ``enters'' or ``stays out'' of a market controlled by a monopolist (Player 2) who can either ``fight'' the entrant or ``share'' the market, and where ``fighting'' is a weakly dominated strategy that yields a strictly worse payoff if the competitor ``enters'' \citep[Ex.~5.4]{Wei95}.
Under the replicator dynamics, ``fight'' becomes extinct if ``enter'' survives (cf. Figure \ref{fig.dom.weak}); however, if Player 1 were to follow a different process under which ``enter'' survives but the integral of its population share over time is bounded, then ``fight'' does not become extinct (cf. the proof of Proposition 3.2 in \citealp{Wei95}).
In higher orders though, the proof of Theorem \ref{thm.dom.weak} goes through for any continuous play $x_{-k}(t)\in\strat_{-k}$, $t\geq0$, of $k$'s opponents, so weakly dominated strategies become extinct independently of how one's opponents evolve over time.
\end{remark}

\begin{remark}
\label{rem.rest}
As noted in Section \ref{sec.dynamics}, starting ``at rest'' is a natural assumption to make from both learning and evolutionary considerations.
First, as far as learning is concerned, this assumption means that players may start with any mixed strategy they wish, but that the learning process \eqref{eq.nLD} is not otherwise skewed towards one strategy or another;
%\footnote{Indeed, $x_{k\alpha}^{(r)}=0$ for all $r=1,\dotsc,n-1$, $\alpha\in\act_{k}$, also implies $y_{k\alpha}^{(r)} = y_{k\beta}^{(r)}$ for all $\alpha,\beta\in\act_{k}$, so being at rest is tantamount to a lack of learning bias.}
similarly, with regards to evolution, starting with $\dot x(0) = \dotso = 0$ is just the baseline of the ``imitation of long-term success'' revision protocol \eqref{eq.success2}.
%in the population interpretation of \eqref{eq.nRD}.
%the assumption that solutions start at rest is simply there to ensure that players do not have an initial learning bias that could skew them to a particular strategy in the absence of external stimuli.%
%\footnote{Indeed, $x_{k\alpha}^{(r)}=0$ for all $r=1,\dotsc,n-1$, $\alpha\in\act_{k}$, also implies $y_{k\alpha}^{(r)} = y_{k\beta}^{(r)}$ for all $\alpha,\beta\in\act_{k}$, so being at rest is tantamount to a lack of learning bias.}
%As such, starting ``at rest'' is a natural assumption to make:
%players may start with any mixed strategy they wish, but it is assumed that the learning process \eqref{eq.nLD} is not otherwise biased towards one strategy or another.

That said, Theorem \ref{thm.dom.weak} still holds if the players' initial velocities (or higher order derivates) are nonzero but small;
if they are too large, weakly dominated strategies may indeed survive.%
\footnote{More precisely, it suffices for the RHS of \eqref{eq.diffdom.weak} to exceed $V_{k}^{(r)}(0)$ for some $t>0$.}
This observation is important for strategies which are only \emph{iteratively} weakly dominated because, if a strategy becomes weakly dominated after removing a strictly dominated strategy, then the system's solutions could approach the face of $\strat$ associated with the resulting restriction of the game with a high velocity towards the newly weakly dominated strategy (e.g. if the iteratively weakly dominated strategy pays very well against the disappearing strictly dominated one; cf. Fig.~\ref{fig.dom.weak}).
Thus, although Theorem \ref{thm.dom.weak} guarantees the elimination of weakly dominated strategies, its conclusions do not extend to iteratively weakly dominated ones.
\end{remark}

\begin{remark}
\label{rem.DF}
A joint application of Theorems \ref{thm.dom.weak} and \ref{thm.dom} reveals that the higher order replicator dynamics \eqref{eq.nRD} perform one round of elimination of weakly dominated strategies followed by the elimination of all strictly dominated strategies.
This result may thus be seen as a dynamic justification of the claim of \cite{DF90} who argue that asking for the iterated deletion of weakly dominated strategies is too strong a requirement for ``rational'' play.

In particular, \citeauthor{DF90} posit that if players are not certain about their opponents' payoffs, then they will not choose a weakly dominated strategy;
however, to proceed with a second round of elimination, players must know that other players will not choose certain strategies, and since weak dominance is destroyed by arbitrarily small amounts of payoff uncertainty, only strictly dominated strategies may henceforth be deleted.
In the same spirit, weakly dominated strategies are eliminated in the higher order replicator dynamics when players begin unbiased;
however, because of the inertial character of \eqref{eq.nLD}, players may develop such a bias over time, so only (iteratively) strictly dominated strategies are sure to become extinct after that phase.%
%\footnote{Interestingly, this reasoning also shows that there is a certain link between the players' learning bias and their uncertainties for the payoffs of other players.}
\end{remark}

\begin{remark}
In a very recent paper, \cite{BHK13} observe a similar behavior in a refined variant of the best reply dynamics of \cite{GM91} where the players' best reply correspondences are modified to include only strategies that are best replies to an open set of nearby states of play.
This refined update process also eliminates weakly dominated strategies, but it requires players to be significantly more informed than in the myopic context of continuous-time deterministic dynamics:
it applies to highly informed, highly rational players who know not only their payoffs at the current state of play, but also their payoffs in all nearby states as well.
Instead, Theorem \ref{thm.dom.weak} shows that weakly dominated strategies (and weakly dominated equilibria) become extinct under much milder information assumptions, namely the players' payoffs at the current state.
\end{remark}

%\begin{remark}
%As with strictly dominated strategies, Theorem \ref{thm.dom.weak} applies to more general imitative dynamics under the same caveats: aggregate-monotonic dynamics eliminate all weakly dominated strategies, convex (resp. concave) monotonic dynamics eliminate pure (resp. mixed) strategies that are weakly dominated by mixed (resp. pure) strategies, and payoff-monotonic dynamics eliminate pure strategies that are weakly dominated by other pure strategies.
%\end{remark}

\begin{remark}
Tying in with Remark \ref{rem.rest} above, we get the following result for weakly dominated Nash equilibria (or for Nash equilibria whose support contains a weakly dominated strategy):
if $q$ is such an equilibrium and players start with sufficiently small learning bias $\dot y(0), \ddot y(0)$, etc., then $\dkl(q\midd x(t)) \to +\infty$.
In particular, there exists a neighborhood $V$ of $q$ in $\strat$ such that every solution orbit $x(t)$ of \eqref{eq.nRD} which starts at rest in $V$ will escape $V$ in finite time, never to return.
In this sense, weakly dominated equilibria are \emph{repelling}, so they may not be selected in the higher order replicator dynamics \eqref{eq.nRD} if players start at rest (see also Theorem \ref{thm.folk} in the following section).
%actually, as we shall see in Theorem \ref{thm.folk} below, only strict equilibria can be attracting in \eqref{eq.nRD}.
\end{remark}

\begin{remark}
Finally, it is important to note that our estimate of the rate of extinction of weakly dominated strategies is one order lower than that of strictly dominated ones;
as a result, Theorem \ref{thm.dom.weak} does not imply the annihilation of weakly dominated strategies in first order dynamics (as well it shouldn't).
Instead, in first order, if there is some adversarial strategy against which the weakly dominant strategy gives a strictly greater payoff than the weakly dominated one, and if the share of this strategy always remains above a certain level, then the weakly dominated strategy becomes extinct (see e.g. Proposition 3.2 in \citealp{Wei95}).
In our higher order setting, this assumption instead implies that weakly dominated strategies become extinct as fast as \emph{strictly} dominated ones:
\end{remark}

\begin{proposition}
\label{prop.dom.weak.unconditional}
Let $x(t)$ be an interior solution of the $n$-th order replicator dynamics \eqref{eq.nRD}, and let $q_{k}\preccurlyeq q_{k}'$.
If there exists $\alpha_{-k}\in\act_{-k}$ with $u_{k}(q_{k};\alpha_{-k}) < u_{k}(q_{k}';\alpha_{-k})$ and $x_{\alpha_{-k}}(t) \geq \eps >0$ for all $t\geq0$, then:
\begin{equation}
\dkl(q_{k}\midd x_{k}(t))
\geq \eps \lambda_{k}
\big[u_{k}(q_{k}';\alpha_{-k}) - u_{k}(q_{k};\alpha_{-k})\big]
t^{n}\big/n! + \bigoh(t^{n-1}).
\end{equation}
%where $\Delta u_{k}^{*} = \min\{u_{k}(q_{k};\beta_{-k}) - u_{k}(q_{k};\beta_{-k}): \beta_{-k}\in\act_{-k}^{*}\}$, and $\act_{-k}^{*}$ is the subset of $\act_{-k} \equiv \prod_{\ell\neq k} \act_{\ell}$ over which the above payoff difference is strictly positive.
\end{proposition}

\begin{proof}
Simply note that the estimate \eqref{eq.diffdom.weak.diff} is bounded from below by $\eps \lambda_{k} u_{k}(q_{k}'-q_{k};\alpha_{-k})$ and follow the same reasoning as in the proof of Theorem \ref{thm.dom}.
\end{proof}

%----------------------------------------------------------------------
%%% THE FOLK THEOREM
%----------------------------------------------------------------------
\section{Stability of Nash play and the folk theorem.}
\label{sec.folk}

In games that cannot be solved by the successive elimination of dominated strategies, one usually tries to identify the game's Nash equilibria instead.
Thus, given the prohibitive complexity of these solutions \citep{DGP06}, one of the driving questions of evolutionary game theory has been to explain how Nash play might emerge over time as the byproduct of a simpler, adaptive dynamic process.

\subsection{The higher order folk theorem.}

A key result along these lines is the \emph{folk theorem of evolutionary game theory} \citep{Wei95,HS88,San10};
for the multi-population replicator dynamics \eqref{eq.1RD}, this theorem can be summarized as follows:
\begin{enumerate}[I.]
\setlength{\itemsep}{0pt}
\setlength{\parskip}{1pt}
\item Nash equilibria are stationary.
\item If an interior solution orbit converges, its limit is Nash.
\item If a point is Lyapunov stable, then it is also Nash.
\item A point is asymptotically stable if and only if it is a strict equilibrium.
\end{enumerate}

Accordingly, our aim in this section will be to extend the above in the context of the higher order dynamics \eqref{eq.nRD}.
To that end however, it is important to recall that the higher order playing field is fundamentally different because the choice of an initial strategy profile $x(0)\in\strat$ does not suffice to determine the evolution of \eqref{eq.nRD};
instead, one must prescribe the full initial state $\omega(0) = (x(0),\dot x(0),\dotsc, x^{(n-1)}(0))$ in the system's phase space $\phase$.
Regardless, a natural way to discuss the stability of initial \emph{points} $q\in\strat$ is via the corresponding \emph{rest states} $(q,0,\dotsc,0)\in\phase$ (recall also the relevant discussion in Section \ref{sec.prelims.dynamics}, Section \ref{sec.dynamics}, and the remarks following Theorem \ref{thm.dom.weak}).
With this in mind, we will say that $q\in \strat$ is \emph{stationary} (resp. \emph{Lyapunov stable}, resp. \emph{attracting}) when the associated rest state $(q,0,\dotsc,0)\in\phase$ is itself stationary (resp. Lyapunov stable, resp. attracting).

In spite of these differences, we have:

\begin{theorem}
\label{thm.folk}
Let $x(t)$ be a solution orbit of the $n$-th order replicator dynamics \eqref{eq.nRD}, $n\geq1$, for a normal form game $\game\equiv \game(\play,\act,u)$, and let $q\in\strat$. Then:
%\vspace{-.25em}
\begin{enumerate}[\upshape I.]
\setlength{\itemsep}{0pt}
\setlength{\parskip}{1pt}

\item
$x(t)=q$ for all $t\geq0$ iff $q$ is a restricted equilibrium of $\game$ (i.e. $u_{k\alpha}(q)=\max\{u_{k\beta}(q): q_{k\beta}>0\}$ whenever $q_{k\alpha}>0$).

\item
If $x(0)\in\intstrat$ and $\lim_{t\to\infty} x(t) = q$, then $q$ is a Nash equilibrium of $\game$.

\item
If every neighborhood $U$ of $q$ in $\strat$ admits an interior orbit $x_{U}(t)$ such that $x_{U}(t)\in U$ for all $t\geq0$, then $q$ is a Nash equilibrium of $\game$.

\item
\label{itm.strict}
Let $q$ be a strict equilibrium of $\game$.
For every neighborhood $U$ of $q$ in $\strat$, there exists a neighborhood $V$ of $q$ in $\strat$ and an open set $W\subseteq\phase$ containing $V\exclude\{(q,0,\dotsc,0)\}$ such that $x(t)\in U$ and $x(t) \to q$ for all initial states $(x(0),\dot x(0),\dotsc)\in W$;
put differently, for every $x\in\strat$ sufficiently close to $q$, we can find a neighborhood $V_{x}$ of $(x,0,\dotsc,0)$ in $\phase$ such that $x(t)$ converges to $q$ for all $(x(0), \dot x(0),\dotsc) \in V_{x}$, but the bound on $\dot x(0)$, $\ddot x(0)$, etc. will depend on $x(0)$.
Conversely, only strict equilibria have this property.
%
%for every face $\strat'$ of $\strat$ that contains $q$, and for every neighborhood $U$ of $q$ in $\strat'$, there exists an open set $V$ of initial states $(x(0),\dot x(0),\dotsc)\in\phase'\equiv\phase(\strat')$ such that $x(t)$ remains in $U$ for all $t$ and converges to $q$ as $t\to\infty$; conversely, only strict equlibria have this property.
\end{enumerate}
As an immediate corollary of \eqref{itm.strict}, we also have:
\begin{enumerate}
\item[\upshape IV$'$.]
If $q$ is a strict equilibrium of $\game$, then it is attracting:
there exists a neighborhood $U$ of $q$ in $\strat$ such that $x(t)\to q$ whenever $x(t)$ starts at rest in $U$ (that is, $x(0)\in U$ and $\dot x(0) = \dotsc = 0$);
conversely, only strict equilibria have this property.
\end{enumerate}
\end{theorem}

The basic intuition behind Theorem \ref{thm.folk} is as follows:
First, stationarity is a trivial consequence of the replicator-like term of the dynamics \eqref{eq.nRD}.
Parts II and III follow by noting that if a trajectory eventually spends all time near a stationary point $q$, then this point must be Nash \textendash\ otherwise, the forces of \eqref{eq.nLD} would drive the orbit away.
Finally, convergence to strict equilibria is a consequence of the fact that they are locally dominant, so the payoff-driven forces \eqref{eq.nLD} point in their direction.
However, before making these ideas precise, it will be important to draw the following parallels between Theorem \ref{thm.folk} and the standard first order folk theorem:

\vspace{.5em}

\noindent
\textbf{Parts I and II}
of Theorem \ref{thm.folk} are direct analogues of the corresponding first order claims;
note however that (II) can now be inferred from (III).

\vspace{.5em}

\noindent
\textbf{Part III}
is a slightly stronger assertion than the standard statement that Lyapunov stability implies Nash equilibrium in first order:
indeed, Lyapunov stability posits that all trajectories which start close enough will remain nearby, whereas Theorem \ref{thm.folk} only asks for one such trajectory.
Actually, this last property is all that is required for the proof of the corresponding part of the first order folk theorem as well, but since there are cases which satisfy the latter property but not the former,%
\footnote{For instance, the equilibrium profile $q_{1} = (1,0)$, $q_{2} = (1/2,1/2)$ of the simple $2\times2$ game with payoff matrices $U_{1} = \mathbf{I}$ and $U_{2}=\mathbf{0}$ is neither Lyapunov stable under the replicator dynamics, nor an $\omega$-limit of an interior trajectory, but it still satisfies the property asserted in Part III of Theorem \ref{thm.folk}.}
we will use this stronger formulation which seems closer to a ``bare minimum'' characterization of Nash equilibria (especially in higher orders).

\vspace{.5em}

\noindent
\textbf{Part IV}
shows that strict equilibria attract all nearby rest states $(x(0),0,\dotsc,0)\in\phase$, but it is not otherwise tantamount to higher order asymptotic stability \textendash\ it would be if $W$ were a neighborhood of $V$ in $\phase$ (or, equivalently, of $(q,0,\dotsc,0)$ in $\phase$) instead of $V\exclude\{(q,0,\dotsc,0)\}$.%
%More precisely, for every nearby point $x\in\strat$, we can find a neighborhood $V_{x}\subseteq\phase$ of initial states that converge to $q$, but there is no \emph{uniform} bound on $\dot x(0),\ddot x(0),$ etc. that ensures convergence to $q$ (see also Fig.~\ref{fig.folk}).%
\footnote{We thank Josef Hofbauer for this remark.}

%This difference between first and higher orders is inextricably tied to the \emph{sine qua non} requirement that any higher order dynamical system needs first and foremost to stay in $\strat$.
%Specifically, for small (but nonzero) $\dot x, \ddot x$, etc., the boundary term \eqref{eq.adjustment} in \eqref{eq.nRD} blows up to infinity as $x_{k\alpha}\to 0$, so the dynamics in the phase space $\phase$ near $(q,0,\dotsc,0)$ will become perpendicular to $\strat$ (as embedded in $\phase$), and this precludes higher order asymptotic stability (see Fig.~\ref{fig.folk}).

This difference between first and higher orders is intimately tied to the bias that higher order initial conditions (such as $\dot y(0)$, $\ddot y(0)$, etc.) introduce in the learning scheme \eqref{eq.nLD}.
More precisely, recall that a nonzero initial score velocity $\dot y(0)$ skews the learning scheme \eqref{eq.nLD} to such an extent that it might end up converging to an arbitrary pure strategy even in the absence of external stimuli
(viz. in a constant game; cf. the relevant discussion at the end of Section \ref{sec.2RD.2st}).
This behavior is highly unreasonable and erratic, so players are more likely to adhere to an unbiased version of \eqref{eq.nLD} with $\dot y(0) = \ddot y(0) = \dotso = 0$.
In that case however, Fa\`a di Bruno's chain rule shows that we will also have $\dot x(0) = \ddot x(0) = \dotso = 0$ in \eqref{eq.nRD}, so Part IV$'$ of Theorem \ref{thm.folk} allows us to recover the first order statement to the effect that all nearby initial strategy choices are attracted to $q$.
Similarly, if we consider the mass-action interpretation of \eqref{eq.nRD} that we put forth in Section \ref{sec.evolution}, players who are unbiased in the calculation of their aggregate payoffs will also have $\dot x(0) = \dots = 0$, so Part IV$'$ of the theorem essentially boils down to the first order asymptotic stability result in that case.

That said, it is also important to note that this convergence statement remains true even if the players' higher order learning bias $\dot y(0), \ddot y(0),\dotsc,$ is nonzero but (uniformly) not too large.
To wit, assume without loss of generality that the strict equilibrium $q$ under scrutiny corresponds to everyone playing their ``0''-th strategy, and consider the associated relative score variables
\begin{equation}
\label{eq.scorediffs}
z_{k\mu} = y_{k\mu} - y_{k,0},\quad
\mu\in\act_{k}^{\ast}\equiv\act_{k}\exclude\{0\}.
\end{equation}
As can be easily seen, these score differences are mapped to strategies $x\in\strat$ via the \emph{reduced Gibbs map}
$G_{k}^{\ast}\from\R^{\act_{k}^{\ast}}\to\strat_{k}$, with $z_{k}\in\R^{\act_{k}^{\ast}} \mapsto G_{k}^{\ast}(z_{k})\in\strat_{k}$ as follows:
\begin{align}
\label{eq.Gibbs.reduced}
\tag{\ref*{eq.Gibbs}$^{\ast}$}
\txs
G_{k,0}^{\ast}(z)
= \left(1 + \insum_{\nu}^{k} e^{z_{k\nu}}\right)^{-1},
&&
\txs
G_{k\mu}^{\ast}(z)
= e^{z_{k\mu}} \left(1 + \insum_{\nu}^{k} e^{z_{k\nu}}\right)^{-1}.
\end{align}
More specifically, if the relative scores $z_{k\mu}$ are given by \eqref{eq.scorediffs}, we will have $G_{k}^{\ast}(z_{k}) = G_{k}(y_{k})$, so the learning scheme \eqref{eq.nLD} will be equivalent to the \emph{relative score dynamics}
\begin{equation}
\label{eq.nZD}
\tag{ZD$_{\text{n}}$}
z_{k\mu}^{(n)} = u_{k\mu}(x) - u_{k,0}(x),
\end{equation}
with (reduced) logit choice $x_{k} = G_{k}^{\ast}(z)$.
In this formulation, the proof of Theorem \ref{thm.folk} shows that if players start with sufficiently negative $z_{k\mu}(0)$ and their learning bias $\dot z_{k\mu}(0), \ddot z_{k\mu}(0),\dotsc,$ does not exceed some \emph{uniform} $M>0$, then the relative scores $z_{k\mu}$ will escape to $-\infty$.
In other words, we will have $x(t)\to q$ whenever $x(0)$ is sufficiently close to $q$ and the players' initial learning bias (which is what players use to update \eqref{eq.nLD} anyway) is uniformly small.%
\footnote{The reason that this reasoning does not apply to nonzero initial strategy growth rates $\dot x(0)$ etc. may be seen in a simple $2$-strategy context as follows:
by \eqref{eq.xzdep} we will have $\dot x = x (1 - x) \dot z$, so a uniform bound on $\dot x$ does not correspond to a uniform bound on $z$.
In particular, if players start with a finite initial score velocity $\dot z(0)$ near the boundary of $\strat$, then this will correspond to a vanishingly small initial strategy growth rate $\dot x(0)$;
conversely, nonzero $\dot x(0)$ with $x(0)$ arbitrarily near $\bd(\strat)$ implies an arbitrarily large initial learning bias $\dot z(0)$.

Similar conclusions apply to the revision formulation \eqref{eq.MD2.der} of the higher order replicator dynamics with offset aggregate payoffs $U_{k\alpha}(t) = U_{k\alpha}(0) + \int_{0}^{t} u_{k\alpha}(x(s)) \dd s$.
Indeed, if the initial offsets $U_{k\alpha}(0)$ are uniformly small in \eqref{eq.MD2.der}, then the corresponding initial conditions $\dot x_{k\alpha}(0)$ will scale with $x_{k\alpha}(0)$;
it thus makes little evolutionary sense to ask for convergence under a \emph{uniform} bound on the players' initial velocities $\dot x_{k\alpha}(0)$ when $x_{k\alpha}(0)$ is itsellf small.}

\begin{remark*}
Using the extended real arithmetic operations for $-\infty$, the reduced Gibbs map \eqref{eq.Gibbs.reduced} maps $(-\infty,\dotsc,-\infty)$ to $q$.
Interestingly, by adjoining $(-\infty,\dotsc,-\infty)$ to $\prod_{k}\R^{\act_{k}^{\ast}}$ in a topology which preserves the continuity of $G_{k}^{\ast}$, the statement above implies that this ``point at negative infinity'' is asymptotically stable in \eqref{eq.nZD} \textendash\ for a detailed statement and proof, see \ref{app.folk}.
\end{remark*}

%\begin{remark}
%As in the first order case, Theorem \ref{thm.folk} applies to more general classes of imitative dynamics. In the context of payoff-monotonic dynamics in particular, our proof goes through unchanged except for the converse implication of Part IV for $n=1$ (see also Theorem \ref{thm.nonconv} below).
%\end{remark}

\begin{figure}[t]
\centering
\subfigure{
\label{subfig.folk.portrait}
\includegraphics[width=0.46\columnwidth]{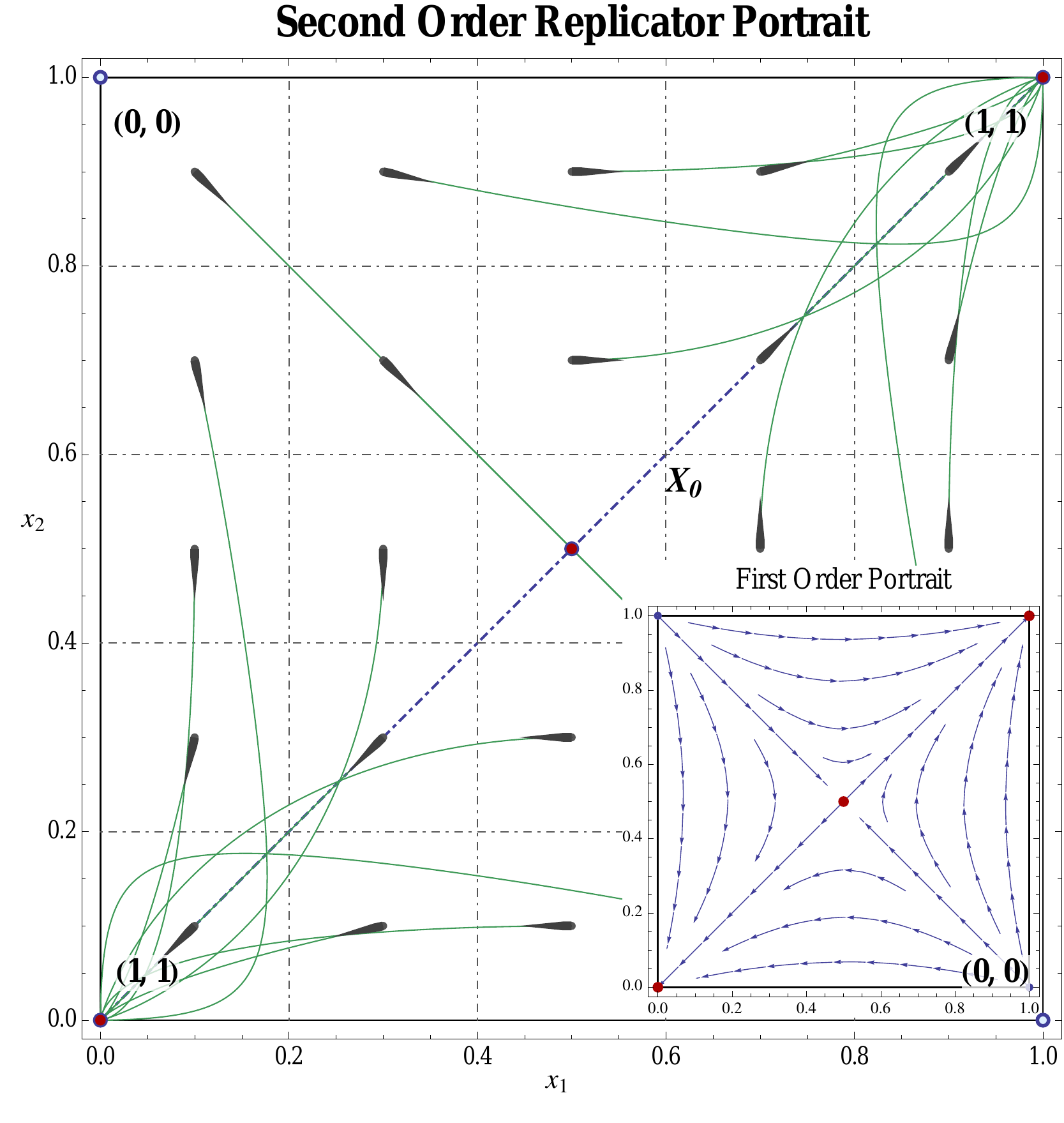}}
\hfill
\subfigure{
\label{subfig.folk.phase}
\includegraphics[width=0.465\columnwidth]{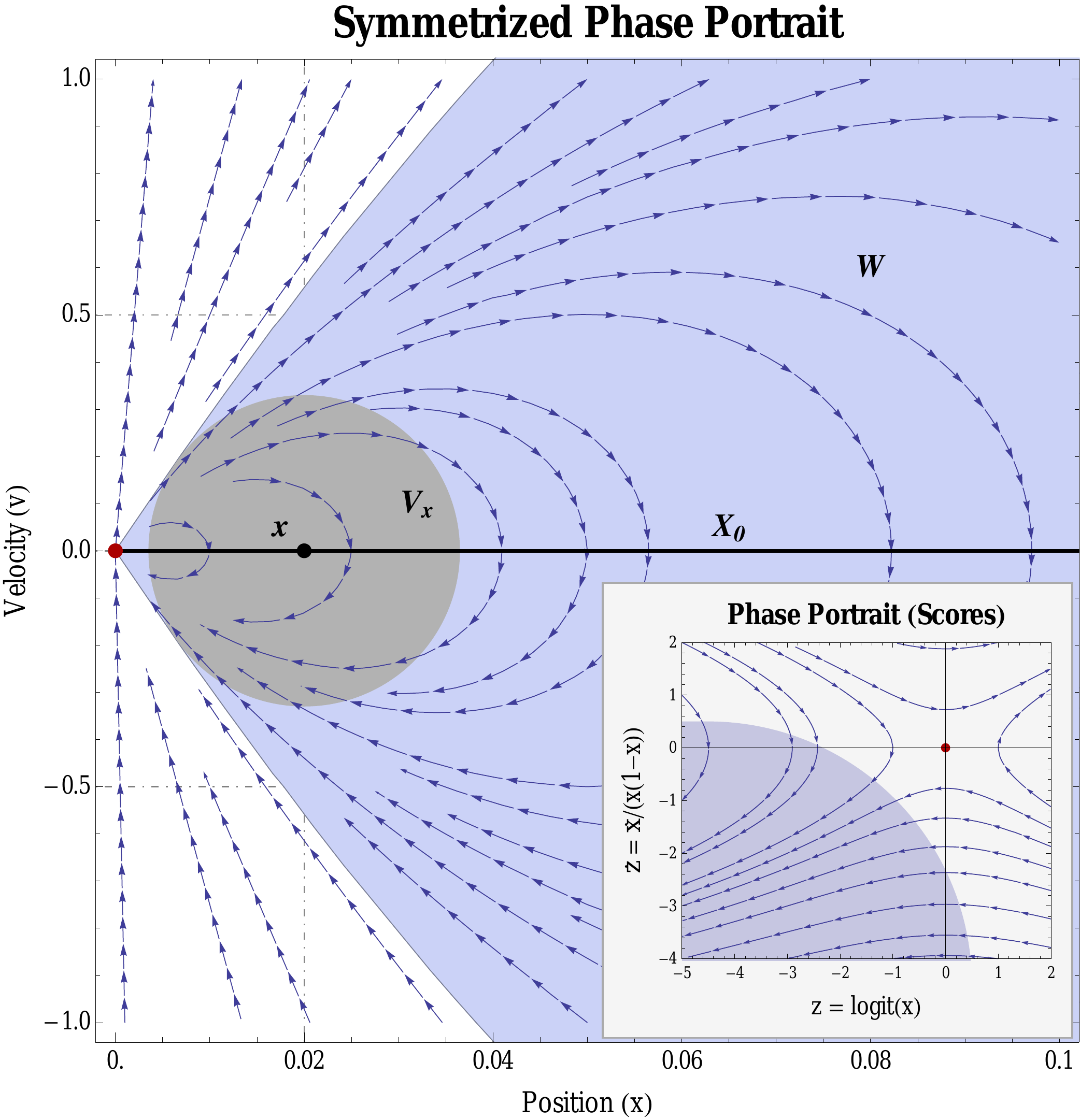}}
%\vspace{-10pt}
\caption{
\footnotesize
Second order replicator trajectories in a $2\times2$ coordination game with payoff matrices $U_{1} = U_{2} = I$.
Fig.~\ref{subfig.folk.portrait} shows the phase portraits for the first and second order replicator dynamics, while Fig.~\ref{subfig.folk.phase} shows the restriction of the game's phase space to the symmetric invariant manifold $\strat_{0}$ which joins the game's equilibria.
For every symmetric initial point $(x,x)$ near $q=(0,0)$, there exists a neighborhood of initial states $V_{x}$ (gray) such that all orbits starting in $V_{x}$ stay close and eventually converge to $q$.
The union $W$ of these neighborhoods (light blue) is not itself a neighborhood of $q$ in $\phase_{0}\equiv\phase(\strat_{0})$, so $q$ is not asymptotically stable in \eqref{eq.nRD};
however, in terms of the score variables $z=\log \frac{x}{1-x}$, $\dot z = \dot x\big/x(1-x)$, the corresponding point at negative infinity $(-\infty,\dotsc,-\infty)$ attracts all nearby initial states (inlay).}
%\vspace{-10pt}
\label{fig.folk}
\end{figure}

To prove Theorem \ref{thm.folk}, we begin with a quick technical lemma:

\begin{lemma}
\label{lem.diffeo}
The reduced Gibbs map $G_{k}^{\ast}\from\R^{\act_{k}^{\ast}}\to\strat_{k}$ of \eqref{eq.Gibbs.reduced} is a diffeomorphism onto its image.
\end{lemma}

\begin{proof}
It is easy to check that the expressions $z_{k\mu} = \log(x_{k\mu}/x_{k,0})$ provide an inverse to $G_{k}^{\ast}$ for $x_{k}\in\mathrm{rel\,int}(X_{k})$;
the claim then follows by noting that all expressions involved are smooth.
\end{proof}

\begin{proof}[Proof of Theorem \ref{thm.folk}]
We will begin with stationarity of restricted equilibria.
Since the payoff term of \eqref{eq.nRD} does not contain any higher order derivatives, it will vanish at $q\in\strat$ if and only if $u_{k\alpha}(q) = u_{k}(q)$ for all $\alpha\in\supp(q)$, implying that $q$ is a restricted equilibrium.
Conversely, let $q$ be a Nash equilibrium in the restriction $\game'\equiv\game(\play,\strat',u_{k}|_{\strat'})$ of $\game$ with $\act_{k}' = \supp(q_{k})$.
Then, with $u_{k\alpha}(q) = u_{k\beta}(q)$ for all $\alpha,\beta\in\act_{k}'$, the updating scheme \eqref{eq.nLD} constrained to $\game'$ and starting at $q$ also gives $y_{k\alpha}^{(n)}(0) = y_{k\beta}^{(n)}(0)$ for all $\alpha,\beta\in\act_{k}'$.
So, if \eqref{eq.nRD} starts at $q$ with initial motion rates $\dot x(0) = \ddot x(0) = \dotsm = 0$, we will have $y_{k\alpha}(t) - y_{k\alpha}(0) = y_{k\beta}(t) - y_{k\beta}(0)$ for all $\alpha,\beta\in\act_{k}'$, and, by the homogeneity of the Gibbs map ($G(y_{1}+c, y_{2}+c,\dotsc) = G(y_{1},y_{2},\dotsc)$ for all $c\in\R$), we readily obtain $x(t) = q$ for all $t$, i.e. $q$ is stationary.%
\footnote{Importantly, Nash equilibria are not stationary in \eqref{eq.nLD}:
orbits that are parallel to the line $(t,\dotsc,t)$ in $\R^{\act}$ are stationary in \eqref{eq.nRD}.}

\smallskip

We now turn to Part (III) of the theorem \textendash\ which will also prove Part (II).
To that end, suppose that every neighborhood $U$ of $q$ in $\strat$ admits an interior orbit $x(t)$ that stays in $U$ for all $t\geq0$;
we then claim that $q$ is Nash.
Indeed, assume instead that for some $k\in\play$, there exists $\beta\in\act_{k}$ and $\alpha\in\supp(q_{k})$ with $u_{k\alpha}(q) < u_{k\beta}(q)$.
Then, pick $\eps>0$ and a neighborhood $U$ of $q$ such that $x_{k\alpha} > q_{k\alpha}/2>0$ and $u_{k\beta}(x) \geq u_{k\alpha}(x) + \eps$ for all $x\in U$.
By assumption, there exists an interior orbit $x(t)$ which stays in $U$ for all time, so, for the associated score variables $y(t)$, we will have:
\[
y_{k\beta}^{(n)}(t) - y_{k\alpha}^{(n)}(t)
= u_{k\beta}(x(t)) - u_{k\alpha}(x(t))
\geq \eps > 0.
\]
This last inequality immediately implies that $\log \big(x_{k\beta}(t)/x_{k\alpha}(t) \big) \to +\infty$, contradicting the fact that $x_{k\alpha}(t)>q_{k\alpha}/2$ for all $t\geq0$.

\smallskip

With regards to Part (IV), let $q = (e_{1,0},\dotsc,e_{N,0})$ be a strict equilibrium of $\game$, and consider the relative scores $z_{k\mu} = y_{k\mu} - y_{k,0}$, $\mu\in\act_{k}^{*} \equiv \act_{k}\exclude\{0\}$ of \eqref{eq.scorediffs}.
Since the reduced Gibbs map $G_{k}^{*}\from\R^{\act_{k}^{*}}\to\strat_{k}$ of \eqref{eq.Gibbs.reduced} is a diffeomorphism onto its image by Lemma \ref{lem.diffeo}, the same will hold for the direct sum $G^{*}\equiv\bigoplus_{k} G_{k}^{*}\from \prod_{k} \R^{\act_{k}^{*}}\to \strat$ as well.
Accordingly, if we take a neighborhood $U_{\eps}$ of $q$ in $\strat$ of the form $U_{\eps} = \{x\in\Int(\strat): x_{k,0} > 1-\eps,\,k\in\play\}$, its preimage under $G^{*}$ will be the set $V_{h} = \{z:\pf_{k,0} < h,\,k\in\play\}$  where $\pf_{k,0} = \insum_{\nu}^{k} \exp(z_{k\nu})$ and $h = (1-\eps)^{-1} - 1$ ($\approx\eps$ for small $\eps$).
We will show that if $h$ is chosen small enough, then there exists $\delta>0$ such that whenever a solution $z(t)$ of \eqref{eq.nZD} starts at $z(0) \in V_{h}$ with $\|z^{(r)}(0)\|< \delta$ for $r=1,\dotsc,n-1$, we will have $z(t) \in V_{2h}$ for all $t\geq0$ and $z_{k\mu}(t)\to -\infty$ for all $\mu\in\act_{k}^{*}$, $k\in\play$.
Since $G^{*}$ is a diffeomorphism onto its image and $x\to q$ iff $z_{k\mu}\to-\infty$ for all $\mu\in\act_{k}^{*}$, $k\in\play$, this will establish the ``if'' direction of our claim.%
\footnote{Non-interior trajectories can be handled similarly by looking at an appropriate restriction $\game'$ of $\game$.}

Indeed, let $z(t)$ be a solution of \eqref{eq.nZD} starting in $V_{h}$ and let $\tau_{2h} = \inf\{t:z(t)\notin V_{2h}\}$ be the time it takes $z(t)$ to escape from $V_{2h}$ (with the usual convention $\inf(\varnothing) = \infty$).
Then, if $h$ is taken small enough, there will be some constant $M>0$ such that $u_{k,0}(x) - u_{k,\mu}(x) \geq M>0$ for all $x\in G^{\ast}(V_{2h})$ (recall that $q$ is a strict equilibrium).
In this way, for $t\leq \tau_{2h}$, Taylor's theorem with Lagrange remainder applied to \eqref{eq.nZD} readily gives:
\begin{equation}
\label{eq.zestimate}
z_{k\mu}(t)
\leq z_{k\mu}(0)
+ \insum_{r=1}^{n-1} z_{k\mu}^{(r)}(0) \, t^{r}\big/r! 
- M t^{n}\big/n!.
\end{equation}
Hence, pick $\delta>0$ such that the maximum of the polynomial $\insum_{r=1}^{n-1} z_{k\mu}^{(r)}(0) \, t^{r}\big/r! - M t^{n}\big/n!$ for $t\geq0$ is strictly smaller than $\log 2$ whenever $|z_{k\mu}^{(r)}(0)|<\delta$ for $\mu\in\act_{k}^{*}$, $k\in\play$, and $r=1,\dotsc n-1$.
This readily yields $\pf_{k,0}(\tau_{2h}) < \insum_{\mu}^{k} \exp\big(z_{k\mu}(0) + \log 2\big) = 2 \pf_{k,0}(0) < 2h$, i.e. $z(\tau_{2h}) \in V_{2h}$, a conclusion which cannot hold unless $\tau_{2h}=\infty$.
We thus obtain $z(t)\in V_{2h}$ for all $t\geq0$, so the limit of \eqref{eq.zestimate} as $t\to\infty$ gives $z_{k\mu}(t)\to-\infty$.
%Then, by using Taylor expansions of a lower order, one can show in a similar fashion that the same also holds for the derivatives $\dot z_{k\mu}(t),\dotsc, z_{k\mu}^{(n-1)}(t)$ as well.

For the converse implication, it is easy to show that any vertex $q$ of $\strat$ which attracts an open neighborhood of initial rest states must also be a strict Nash equilibrium: extending the reasoning of \citet[Thm.~1]{RW95} to our higher order setting, it suffices to consider the evolution of the dynamics in the edge which joins $q=(\alpha_{k};\alpha_{-k})$ to a vertex $q' = (\alpha_{k}';\alpha_{-k})$ with $u_{k}(q') \geq u_{k}(q)$. However, Theorem \ref{thm.nonconv} shows that only a vertex $q\in\strat$ can attract an open set of initial states $\omega\in\phase$ containing a punctured neighborhood of $q$ in $\strat$, so our assertion follows.
\end{proof}

Now, with regards to the equilibration speed of the higher order dynamics, it can be shown that the rate of convergence to a strict equilibrium in the $n$-th order dynamics \eqref{eq.nRD} is $n$ orders as fast as in the first order regime.
More specifically, we have:
\begin{proposition}
\label{prop.convrate.strict}
Let $q = (e_{1,0},\dotsc,e_{N,0})$ be a strict Nash equilibrium of the finite game $\game$, and let $x(t)$ be a solution orbit of the replicator dynamics \eqref{eq.nRD} which starts at rest and close enough to $q$.
Then, there exists a positive constant $c>0$ such that:
%with linear dependence on the payoff differences $u_{k,0}(q) - u_{k\mu}(q)$, $\mu\in\act_{k}^{\prime}\equiv\act_{k}\exclude\{0\}$ such that:
\begin{equation}
\label{eq.rate.strict}
x_{k,0}(t) \sim 1- \exp\big(-c t^{n}/n! + \bigoh(t^{n-1})\big).
\end{equation}
\end{proposition}

\begin{proof}
%By Theorem \ref{thm.folk}, we know that if $x(t)$ starts close enough to $q$ and is initially at rest, then it will always remain close to $q$.
If we choose a sufficiently small neighborhood of initial positions, Theorem \ref{thm.folk} shows that the payoff differences $u_{k,0}(x(t)) - u_{k,\mu}(x(t))$ will be bounded away from $0$ by some positive constant $c>0$ for all $\mu\in\act_{k}^{*}$, $k\in\play$ and for all $t\geq0$.
Hence, with $z_{k\mu}^{(n)} \leq -c<0$ by \eqref{eq.nZD}, our assertion follows from an $(n-1)$-fold application of the mean value theorem.
\end{proof}

\subsection{Dynamic instability of mixed equilibria.}

Theorem \ref{thm.folk} and Proposition \ref{prop.convrate.strict} above characterize the behavior of the $n$-th order replicator dynamics near strict equilibria from both a qualitative and a quantitative viewpoint;
on the flip side, they do not address mixed equilibria.
To study this issue, recall first that the standard asymmetric replicator dynamics preserve a certain volume form in the interior of $\strat$, so mixed equilibria cannot be attracting in first order.
\cite{RW95} establish this ``incompressibility'' property of the replicator dynamics by taking an ingenious extrinsic re\-pa\-ra\-me\-tri\-zation which makes the replicator dynamics divergence-free in the interior of $\strat$ (see also \citealp{RV90}).
On the other hand, by exploiting the interplay between the natural logarithm and the Gibbs map \eqref{eq.Gibbs}, \cite{Hof96} essentially showed that the replicator dynamics are incompressible in the space of the score variables $y_{k\alpha}$.
%In first order, this idea can be used to show that the generalized imitative dynamics \eqref{eq.nGD} are volume-preserving whenever the payoff observables $w_{k\alpha}$ do not depend on $x_{k\alpha}$;
%in higher orders however, the dynamics are further decoupled because the $w_{k\alpha}$ are only tied to the players' mixed strategies and not their velocities.
In the same spirit, we have:

\begin{proposition}
\label{prop.incompressibility}
The flow of the higher order learning dynamics \eqref{eq.nLD} preserves volume in the usual Euclidean geometry of $\R^{\act}$ for all $n\geq1$;
the same holds for \eqref{eq.nRD} w.r.t. a non-Euclidean volume form on the system's phase space $\phase$.
\end{proposition}

\begin{proof}
Rewriting \eqref{eq.nLD.nsteps} in continuous time, \eqref{eq.nLD} is equivalent to the first order system:
\begin{flalign}
\label{eq.nGLD.order1}
\dot y_{k\alpha}^{n-1}(t) &= u_{k\alpha}(x(t)),\notag\\
			&\dotsm\\
%\dot y_{k\alpha}^{1}(t)	&= y_{k\alpha}^{2}(t),\notag\\
\dot y_{k\alpha}^{0}(t)	&= y_{k\alpha}^{1}(t).\notag
\end{flalign}
Thus, given that $y_{k\alpha}^{r}$ does not appear in the equation for $\dot y_{k\alpha}^{r}$ for $r=0,\dotsc,n-1$, it follows that the flow of \eqref{eq.nLD} will be incompressible in the standard Euclidean metric of $\R^{\act}$.
%\footnote{This remains true for $n=1$ whenever $\pd w_{k\alpha}/\pd x_{k\alpha} = 0$, i.e. for the replicator dynamics (or, more generally, if $\Div w = 0$).}
Using the relative scores $z$ of \eqref{eq.scorediffs}, the same argument applies to the dynamics \eqref{eq.nZD}, and since $G^{*}$ is a diffeomorphism onto its image by Lemma \ref{lem.diffeo}, the result carries over to \eqref{eq.nRD} as well.
\end{proof}

Thanks to this incompressibility property of \eqref{eq.nLD} and \eqref{eq.nRD}, we have:

\begin{theorem}
\label{thm.nonconv}
In the higher order replicator dynamics \eqref{eq.nRD}, interior points cannot attract open sets of initial states;
only vertices of $\strat$ can be attracting.
More generally, a non-pure point $q\in\strat$ can only attract relatively open sets of initial states whose support in $\strat$ properly contains that of $q$.
\end{theorem}

\begin{proof}
We will prove that if $q\in\Int(\strat)$, then there is no open set of initial conditions in $\phase$ that converges to $q$.
The result for general non-pure $q\in\strat$ will then follow by focusing on the face $\strat'$ of $\strat$ which is spanned by the support of $q$, i.e. $\strat'=\prod_{k} \simplex(\act_{k}')$ with $\act_{k}' = \supp(q_{k})$;
since the dynamics \eqref{eq.nRD} preserve the faces of $\strat$, the assertion follows by noting that the intersection of $\strat'$ with an open set in $\strat$ is open in $\strat'$ by definition.

Working with the variables $z$ of \eqref{eq.scorediffs} and recalling that the map $G^{*}\from z \mapsto x$ is a diffeomorphism onto its image by Lemma \ref{lem.diffeo}, Proposition \ref{prop.incompressibility} shows that open sets of initial states in $\phase$ cannot converge to the interior state $((G^{*})^{-1}(q),0,\dotsc,0)$.
Thus, to establish the theorem's claim that $z(t)$ cannot converge to the interior point $z^{*}\equiv(G^{*})^{-1}(q)$, it suffices to show that if $z(t)\to z^{*}$, then we would also have $\lim_{t\to\infty}\dot z(t) = \lim_{t\to\infty} \ddot z(t) = \dotsm = 0$.

For notational simplicity, we will only prove the case $n=2$.
To that end, assume for the purposes of establishing a contradiction that $z_{k\mu}(t) \to z_{k\mu}^{*}$ for some $\mu\in\act_{k}^{*}$, $k\in\play$, but that $\dot z_{k\mu}(t)\nrightarrow 0$.
Then, without loss of generality, there exists $\eps>0$ and an increasing sequence of times $t_{n}\to \infty$ such that $\dot z_{k\mu}(t_{n})\geq \eps$ for all $n$.
Hence, let $J_{n}$ be the largest open interval which contains $t_{n}$ and which is such that $\dot z_{k\mu} > \eps/2$ in $J_{n}$;
we then claim that the length $\delta_{n} = m(J_{n})$ of $J_{n}$ vanishes as $n\to\infty$.
Indeed, by passing to a subsequence of $t_{n}$ if necessary, assume that $\delta_{n}$ always exceeded some positive $\delta>0$;
then, with $\dot z_{k\mu}> \eps/2$ in $J_{n}$, it follows that $z_{k\mu}(t)$ would grow by at least $\eps\delta/2$ over $J_{n}$ for all $n$,
but since $z_{k\mu}(t)$ converges, every subsequence of $z_{k\mu}(t)$ must also be Cauchy, a contradiction.
%all subsequences of $z_{k\mu}(t)$ are Cauchy (recall that $z_{k\mu}(t)\to0$), so if $\delta_{n}$ always exceeded some positive $\delta>0$, then, by passing to a subsequence of $t_{n}$ if necessary, $z_{k\mu}(t)$ would grow by at least $\eps\delta/2$ over $J_{n}$ for all $n$, a contradiction.
Then, by the definition of $J_{n}$, we will have $\dot z_{k\mu} > \eps$ at some interior point of $J_{n}$ and $\dot z_{k\mu} = \eps/2$ at its endpoints;
thus, by the mean value theorem, there exists some $\xi_{n}\in J_{n}$ with $\ddot z_{k\mu}(\xi_{n}) \geq \eps\big/2\delta_{n}$, and hence, $\ddot z_{k\mu}(\xi_{n})\to +\infty$.
However, since $z_{k\mu}^{*}$ must also be a rest point of \eqref{eq.nZD}, the dynamics \eqref{eq.nZD} give $\ddot z_{k\mu}(t)\to 0$ as $t\to\infty$, a contradiction.
\end{proof}

The property that only vertices of $\strat$ can be attracting in the higher order replicator dynamics \eqref{eq.nRD} directly mirrors the first order case.
In the following section however, we will show that this is a property of a much more general class of higher order dynamics, so higher order considerations actually sharpen the instability of non-pure equilibria.

%----------------------------------------------------------------------
%%% EXTENSIONS
%----------------------------------------------------------------------
\section{Extensions: imitative and payoff-monotonic dynamics.}
\label{sec.extensions}

In this section, our aim is to provide several extensions of the higher order dynamics \eqref{eq.nLD} and \eqref{eq.nRD} and to show how the rationality analysis of the previous sections applies to this more general setting.
To that end, if players do not base the updating \eqref{eq.nLD} of their performance scores on the payoffs $u_{k\alpha}(x)$ of the game but on a different set of ``payoff observables'' $w_{k\alpha}\from\strat\to\R$ (assumed continuous), then we obtain the generalized reinforcement scheme
\begin{equation}
\label{eq.nGLD}
\tag{GLD$_{n}$}
y_{k\alpha}^{(n)} = w_{k\alpha}(x),
\end{equation}
which, coupled with the logit choice model \eqref{eq.Gibbs}, yields the generalized $n$-th order dynamics:
\begin{equation}
\label{eq.nGD}
\tag{GD$_{n}$}
x_{k\alpha}^{(n)}
= \lambda_{k} x_{k\alpha} \left(w_{k\alpha}(x) - \insum_{\beta}^{k} x_{k\beta} w_{k}(x)\right)
- x_{k\alpha} \big(R_{k\alpha}^{n-1} - \insum_{\beta}^{k} x_{k\beta} R_{k\beta}^{n-1}\big).
\end{equation}
%where $w_{k}$ is the player average $w_{k}(x) \equiv \insum_{\alpha}^{k} x_{k\alpha} w_{k\alpha}(x)$ (and similarly for $R_{k}$).

The dynamics \eqref{eq.nGD} are characterized by the property that if $x(0)$ lies in a subface $\strat'$ of $\strat$, then $x(t)$ will remain in $\strat'$ for all time:
in other words, if the strategy share $x_{k\alpha}$ of a pure strategy $\alpha\in\act_{k}$ is initially zero, then it remains zero for all time (see also Remark \ref{rem.invariance} at the end of Section \ref{sec.XL}).
This invariance property is known as ``imitation'' \citep{Wei95}, so the dynamics \eqref{eq.nGD} may be seen as a higher order extension of the class of \emph{imitative game dynamics} introduced by \cite{BW96}:
in particular, \eqref{eq.nGD} is the higher order extension of the general imitative equation $\dot x_{k\alpha} = x_{k\alpha} (w_{k\alpha} - \insum_{\beta}^{k} x_{k\beta} w_{k\beta})$ in the same way that \eqref{eq.nRD} extends the standard replicator dynamics \eqref{eq.1RD} to an $n$-th order setting.

Of course, if the payoff observables $w_{k\alpha}$ are not correlated with the game's payoffs $u_{k\alpha}$, the dynamics \eqref{eq.nGD} will not lead to any sort of rational play over time.
To account for this, \cite{SZ92} considered the aggregate-monotonicity criterion
\begin{equation}
\label{eq.aggregate.monotonicity}
\tag{AM}
w_{k}(q_{k}';x_{-k}) > w_{k}(q_{k};x_{-k})\;
\text{ if and only if }\;
u_{k}(q_{k}';x_{-k}) > u_{k}(q_{k};x_{-k}),
\end{equation}
with $x_{-k}\in\strat_{-k}\equiv\prod_{\ell\neq k}\strat_{\ell}$ and $q_{k}, q_{k}'\in\strat_{k}$.
Accordingly, following \cite{HW96}, we will say that the higher order dynamics \eqref{eq.nGD} are:
%\vspace{-.25em}
\begin{itemize}
\setlength{\itemsep}{0pt}
\setlength{\parsep}{0pt}
\setlength{\parskip}{1pt}
\item
\emph{aggregate-monotonic} if \eqref{eq.aggregate.monotonicity} holds for all $q_{k}, q_{k}'\in\strat_{k}$.
\item
\emph{convex-monotonic} when the ``if'' direction of \eqref{eq.aggregate.monotonicity} holds for all pure $q_{k}'$.
\item
\emph{concave-monotonic} when the ``if'' direction of \eqref{eq.aggregate.monotonicity} holds for all pure $q_{k}$.
\item
\emph{payoff-monotonic} if \eqref{eq.aggregate.monotonicity} holds for $q_{k}$ and $q_{k}'$ that are \emph{both} pure.
\end{itemize}
%\vspace{-.25em}

In the first order regime, \cite{SZ92} showed that payoff-mo\-no\-to\-nic (resp. aggregate-monotonic) dynamics eliminate all pure (resp. mixed) dominated strategies.
This result was extended by \cite{HW96} to pure strategies which are dominated by mixed ones in convex-monotonic dynamics, while \cite{Vio11} recently established the dual result for concave dynamics.
In the same spirit, the rationality analysis of Sections \ref{sec.dominance} and \ref{sec.folk} yields:

\begin{proposition}
\label{prop.dom.monotonic}
For any interior initial condition, we have:
%\vspace{-.25em}
\begin{itemize}
\setlength{\itemsep}{0pt}
\setlength{\parsep}{0pt}
\setlength{\parskip}{1pt}
\item
Aggregate-monotonic $n$-th order dynamics eliminate all dominated strategies.
\item
Convex (resp. concave) monotonic $n$-th order dynamics eliminate all pure (resp. mixed) strategies that are dominated by mixed (resp. pure) strategies.
\item
Payoff-monotonic $n$-th order dynamics eliminate all pure strategies that are dominated by pure strategies.
\end{itemize}
%\vspace{-.25em}
If, in addition, players start at rest ($\dot x(0) = \dotso = x^{(n-1)}(0) = 0$), then the above conclusions hold with the characterization ``dominated'' replaced by ``weakly dominated''.
Finally, the rate of extinction is exponential in $t^{n}$ (or $t^{n-1}$ for weakly dominated strategies) in the sense of \eqref{eq.domrate.mixed}/\eqref{eq.domrate.weak}.
\end{proposition}

\begin{proof}
The crucial point in the proof of Theorems \ref{thm.dom} (resp.~Theorem \ref{thm.dom.weak}) is the lower bound for the $n$-th (resp. $(n-1)$-th) derivative of the difference $V_{k}(x) = \dkl(q_{k}\midd x_{k}) - \dkl(q_{k}' \midd x_{k})$ which determines the rate of extinction of dominated (resp.~weakly dominated) strategies.
Thus, by replacing $u$ by $w$ in \eqref{eq.diffdom} (resp. \eqref{eq.diffdom.weak}), and using the appropriate monotonicity condition for each case of dominance (pure/mixed by pure/mixed), our assertion follows along the same lines as Theorem \ref{thm.dom} (resp.~Theorem \ref{thm.dom.weak}).
\end{proof}

In the same spirit, we obtain the following counterpart to the higher order folk theorem for higher order payoff-monotonic dynamics:

\begin{proposition}
\label{prop.folk.monotonic}
The conclusions of Theorem \ref{thm.folk} hold for all higher order ($n\geq2$) payoff-monotonic dynamics.
\end{proposition}

This proposition follows by replacing $u$ with $w$ in the proof of Theorem \ref{thm.folk} and using the payoff-monotonicity condition \eqref{eq.aggregate.monotonicity}, so there are no qualitative differences between first and higher order payoff-monotonic dynamics.
%other than that however, the situation is qualitatively the same as in the case of the higher order replicator dynamics.
On the other hand, Proposition \ref{prop.incompressibility} and Theorem \ref{thm.nonconv} hold for a much wider class of higher order dynamics:

\begin{proposition}
\label{prop.nonconv.imit}
The flow of the generalized learning dynamics \eqref{eq.nGLD}, $n\geq2$, is volume-pre\-ser\-ving;
moreover, the same holds for \eqref{eq.nGD} w.r.t. a non-Euclidean volume form on $\strat$.
Consequently, the conclusions of Theorem \ref{thm.nonconv} hold for all higher order imitative dynamics of the form \eqref{eq.nGD}.
\end{proposition}

\begin{proof}
Since the payoff observables $w_{k\alpha}$ do not depend on $\dot x$ and other higher order derivates, incompressibility stems from the equivalent first order formulation \eqref{eq.nGLD.order1} of \eqref{eq.nGLD} with $u$ replaced by $w$;
the conclusions of Theorem \ref{thm.nonconv} are then proved in the same way.
\end{proof}

Interestingly, if $n=1$ and the payoff observables $w_{k\alpha}$ do not depend on $x_{k\alpha}$, then the system \eqref{eq.nLD.nsteps} remains divergence-free and the conclusions of Proposition \ref{prop.nonconv.imit} continue to apply.
For instance, this explains why the asymmetric replicator equation is divergence-free whereas its symmetric counterpart isn't:
in the case of the former, we have $w_{k\alpha}(x) = u_{k\alpha}(x) = u_{k}(\alpha;x_{-k})$, a quantity which is independent of $x_{k\alpha}$;
in the symmetric case however, if $U$ denotes the payoff matrix of the game being played, then we would have $w_{\alpha}(x) = u_{\alpha}(x) = u(\alpha;x) = \insum_{\beta} U_{\alpha\beta} x_{\beta}$, showing that, in general, the symmetrized replicator dynamics are not divergence-free.%
%\footnote{That said, note that $\frac{\pd x_{\beta}}{\pd y_{\alpha}} = x_{\alpha}(\delta_{\alpha\beta} - x_{\beta})$, so taking the divergence of \eqref{eq.nLD.nsteps} for $n=1$ and $u_{\alpha} = \insum_{\beta} U_{\alpha\beta} x_{\beta}$ yields the quantity
%$T(x)
%= \insum_{\alpha} U_{\alpha\alpha} x_{\alpha} - \insum_{\alpha,\beta} x_{\alpha} U_{\alpha\beta} x_{\beta}
%= \insum_{\alpha,\beta} x_{\alpha} \left(U_{\alpha\alpha} - U_{\alpha\beta}\right) x_{\beta}$.}
% if and only if $\tr(U) = \insum_{\beta} U_{\beta\beta} = 0$

In view of the above, Proposition \ref{prop.nonconv.imit} clashes quite strongly with the first order regime.
For instance, if we take Maynard Smith's payoff-adjusted variant of the replicator dynamics (whereby players divide \eqref{eq.1RD} by their average payoffs), then there exist games with asymptotically stable interior equilibria (for instance, see the Matching Pennies example of \citealp{Wei95}).
In higher orders however, Proposition \ref{prop.nonconv.imit} shows that this can no longer be the case:
the learning dynamics \eqref{eq.nGLD} endow orbits with a tangential acceleration component, and this acceleration carries them away from interior equilibria and towards the boundary of $\strat$.
As a result, only vertices of $\strat$ can be attracting in higher order dynamics of the general class \eqref{eq.nGD}.

%----------------------------------------------------------------------
%%% CONCLUSIONS
%----------------------------------------------------------------------
\section{Concluding remarks.}

The results in the present paper suggest that higher order considerations open the door to some intriguing new questions and directions in the study of learning and evolution in games.
For one, the elimination of weakly dominated strategies is a key feature of higher order dynamics which puts them firmly apart from all their first order siblings;
coupled with the survival of \emph{iteratively} weakly dominated strategies, this provides a dynamic justification of the well-known \sw rationalizability process of \cite{DF90} which cannot otherwise arise from first order considerations.
Furthermore, the population interpretation of our higher order dynamics by means of ``long-term'' variants of existing revision protocols paves the way to a wide array of new classes of dynamics where the impossibility theorem of \cite{HMC03} no longer bars the way \textendash\ a point also made by \cite{SA05} in the context of derivative action fictitious play algorithms.

Nevertheless, even before considering other classes of higher order dynamics, several important questions remain:
For instance, are the higher order replicator dynamics consistent (e.g. as in \citealp{Sor09})?
What can we expect in symmetric, single-population environments (where payoffs are no longer multilinear) or with respect to setwise solution concepts \textendash\ such as sets that are closed under better replies \citep{RW95}?
%More generally, what can we expect if we move beyond a dynamical framework altogether and replace the learning process \eqref{eq.score.int} with a more general \emph{integral} equation of the form $Y_{k\alpha}(t) = \int_{0}^{t} \phi(t-s) u_{k\alpha}(x(s)) \id s$ where the ``learning kernel'' $\phi$ describes the weight that players assign to their past observations?
%The $n$-th order replicator dynamics can be shown to correspond to monomial learning kernels of the form $\phi(t) = t^{n}$, but how would learning be affected by e.g. an exponential discounting (or reinforcement) of the past?
Finally, from the point of view of learning, our approach has been focused on continuous time with players being able to observe (or otherwise calculate) the payoffs associated to their mixed strategies.
This last assumption is relatively harmless in a nonatomic population setting, but crucial from an atomic point of view;
in particular, it is only natural to ask whether our results continue to apply in discrete-time environments with a finite number of players only being able to observe their in-game payoffs.
%This choice has been a conscious one and it was motivated by the fact that our goal was simply to illustrate the rationality properties of the limiting continuous-time dynamics;
%the subtleties however of the descent from the continuous to the discrete (and there are many), so we leave this issue as a future direction (the papers by \citealp{Rus99}, and \citealp{Sor09}, may serve as an indication of what to expect in the discrete case).
%In a similar vein, other important extensions concern symmetric, single-po\-pu\-lation environments where the non-mul\-ti\-li\-near nature of the players' payoffs functions promises significantly different results (especially as far as interior evolutionarily stable strategies are concerned);
%the study of setwise concepts (such as sets that are closed under better replies as in \citealp{RW95});
%the analysis of innovative dynamics \citep{HS11} which do not leave the faces of the game's strategy space invariant, and thus cannot be captured by the exponential mapping \eqref{eq.Gibbs} (at least in the boundary of the strategy space);
%etc.
%Needless to say however, these are all directions that would take much more than a single paper to explore, so we chose to exhibit here only some of the most prominent features and subtleties of the higher order landscape, deferring these questions to future investigations.

\section*{Acknowledgments.}

This paper has benefited greatly from the insightful comments of two anonymous referees and the associate editor who prompted us to include the discussion of Section \ref{sec.evolution}.
We are also grateful to Josef Hofbauer for pointing out an important issue in an earlier formulation of Theorem \ref{thm.folk} and to Chris\-tina Paw\-lo\-witsch and Yannick Viossat for many discussions on the topic;
finally, we would also like to thank the audiences of the Paris Game Theory Seminar, the Paris Working Group on Evolutionary Games, and the 2011 International Conference on the Mathematical Aspects of Game Theory and its Applications for their comments and penetrating questions (Bill Sandholm and J\"orgen Weibull in particular).

%This work was carried out when the second author was visiting the Economics Department of \'Ecole Polytechnique, Paris, France, and we would like to thank the \'Ecole Polytechnique for its generous support and hospitality during 2010-11. The authors also gratefully acknowledge financial support from the Agence Nationale de la Recherche (grant ANR-10-BLAN 0112) and the GIS ``Sciences de la d\'ecision''.

%*************************************************************
%*****    APPENDICES
%*************************************************************

\appendix

\section{Asymptotic stability in terms of relative scores.}
\label{app.folk}

Our aim in this appendix will be to make precise sense of the asymptotic stability statement for the dynamics \eqref{eq.nZD} in Section \ref{sec.folk};
also, to simplify notation, we will drop the index $k$ and rely on context to resolve any ambiguities.

To begin with, let $\R' = \R\cup\{-\infty\}$ denote the real number line extended to one end by adjoining $-\infty$.
Formally, we will use all extended real number operations for $-\infty$, and $\R'$ will be made into a topological space by defining the basic neighborhoods of $-\infty$ to be all sets of the form $U_{a} = \{x\in\R': x<a\}$;
by taking the product topology, we will then form the extended product space $Z = \R^{\act^{\ast}}\cup\{\neginf\}$ by adjoining the ``point at negative infinity'' $\neginf\equiv(-\infty,\dotsc,-\infty)$ to $\R^{\act^{\ast}}$.
In this way, the reduced Gibbs map \eqref{eq.Gibbs.reduced} may be extended to $Z$ by mapping $\neginf$ to $q\in\strat$, and, by the topology of $Z$, this extension will be continuous.

To incorporate initial conditions at negative infinity for \eqref{eq.nZD}, we will work with the extended configuration space $Z$ and the similarly extended phase space $\Omega' = \Omega\cup\left(\{\neginf\}\times Z\times\dotsm\times Z\right)$, where $\Omega$ denotes the original phase space of \eqref{eq.nZD}.
To extend the flow of \eqref{eq.nZD} to all of $\Omega'$, we will define trajectories starting at $\neginf$ to have $z(t) = \neginf$ for all $t\geq0$;
then, using the extended real number operations for $-\infty$ if needed, we will have $z_{\mu}(t) = z_{\mu}(0) +\dot z_{\mu}(0) + \frac{1}{2}\ddot z_{\mu}(0) t^{2} + \dotsb + \int \dotsi \int \Delta u_{\mu}(x(t_{1})) \dd t_{1} \dotsm d t_{n}$ for all $t_{n}\equiv t>0$ and for all $(z(0), \dot z(0),\dotsc)\in\Omega'$, where $x(s) = G^{\ast}(z(s))$.
As a consequence of the above, it is then easy to see that the above collection of trajectories indeed defines a continuous flow on $\Omega'$ which reduces to the flow of \eqref{eq.nZD} on $\Omega$.

With regards to asymptotic stability, Theorem \ref{thm.folk} shows that if all initial conditions $z_{k\mu}(0)$, $\dot z_{k\mu}(0)$ are uniformly small (say, less than some $a\in\R$), then we will have $z_{\mu}(t) \to -\infty$;
moreover, an easy adaptation of the proof of Theorem \ref{thm.folk} shows that the same will hold for all derivates $\dot z_{\mu}(t), \ddot z_{\mu}(t),\dotsc$ of $z(t)$ as well.
This shows that if $z(t)$ starts at a neighborhood of $\{\neginf\}\times\{\neginf\}\times\dotsm$ in $\Omega'$, then it will converge there;
completing the argument for Lyapunov stability as in the proof of Theorem \ref{thm.folk}, we thus obtain:

\begin{proposition*}
With respect to the extended real number topology defined above, the state $\{\neginf\}\times\{\neginf\}\times\dotsm$ which corresponds to the strict equilibrium $q$ is asymptotically stable in \eqref{eq.nZD}.
\end{proposition*}

%*************************************************************
%*****    BIBLIOGRAPHY
%*************************************************************

\begingroup
\addtolength{\bibsep}{3pt}
\bibliographystyle{chicago}
\bibliography{Bibliography}

\begin{thebibliography}{37}
\providecommand{\natexlab}[1]{#1}
\providecommand{\url}[1]{\texttt{#1}}
\providecommand{\urlprefix}{URL }
\expandafter\ifx\csname urlstyle\endcsname\relax
  \providecommand{\doi}[1]{doi:\discretionary{}{}{}#1}\else
  \providecommand{\doi}{doi:\discretionary{}{}{}\begingroup
  \urlstyle{rm}\Url}\fi
\providecommand{\eprint}[2][]{\url{#2}}

\bibitem[{Alvarez(2000)}]{Alv00}
Alvarez, F., 2000: On the minimizing property of a second order dissipative
  system in {Hilbert} spaces. \textit{SIAM Journal on Control and
  Optimization}, \textbf{38~(4)}, 1102--1119.

\bibitem[{Attouch et~al.(2000)Attouch, Goudou, and Redont}]{AGR00}
Attouch, H., X.~Goudou, and P.~Redont, 2000: The heavy ball with friction
  method, {I}. {T}he continuous dynamical system: global exploration of the
  local minima of a real-valued function by asymptotic analysis of a
  dissipative dynamical system. \textit{Communications in Contemporary
  Mathematics}, \textbf{2~(1)}, 1--34.

\bibitem[{Balkenborg et~al.(2013)Balkenborg, Hofbauer, and Kuzmics}]{BHK13}
Balkenborg, D., J.~Hofbauer, and C.~Kuzmics, 2013: Refined best reply
  correspondence and dynamics. \textit{Theoretical Economics}, \textbf{8},
  165--192.

\bibitem[{Bj\"ornerstedt and Weibull(1996)}]{BW96}
Bj\"ornerstedt, J. and J.~W. Weibull, 1996: {Nash} equilibrium and evolution by
  imitation. \textit{The Rational Foundations of Economic Behavior}, K.~J.
  Arrow, E.~Colombatto, M.~Perlman, and C.~Schmidt, Eds., St. Martin's Press,
  New York, NY, 155--181.

\bibitem[{Cantrell(2000)}]{Can00}
Cantrell, C.~D., 2000: \textit{Modern mathematical methods for physicists and
  engineers}. Cambridge University Press, Cambridge, UK.

\bibitem[{Daskalakis et~al.(2006)Daskalakis, Goldberg, and
  Papadimitriou}]{DGP06}
Daskalakis, C., P.~W. Goldberg, and C.~Papadimitriou, 2006: The complexity of
  computing a {Nash} equilibrium. \textit{{STOC} '06: Proceedings of the 38th
  annual ACM symposium on the Theory of Computing}.

\bibitem[{Dekel and Fudenberg(1990)}]{DF90}
Dekel, E. and D.~Fudenberg, 1990: Rational behavior with payoff uncertainty.
  \textit{Journal of Economic Theory}, \textbf{52}, 243--267.

\bibitem[{Fl\r{a}m and Morgan(2004)}]{FM04}
Fl\r{a}m, S.~D. and J.~Morgan, 2004: Newtonian mechanics and {Nash} play.
  \textit{International Game Theory Review}, \textbf{6~(2)}, 181--194.

\bibitem[{Fraenkel(1978)}]{Fra78}
Fraenkel, L.~E., 1978: Formul\ae\ for higher derivatives of composite
  functions. \textit{Mathematical Proceedings of the Cambridge Philosophical
  Society}, \textbf{83~(2)}, 159--165.

\bibitem[{Fudenberg and Harris(1992)}]{FH92}
Fudenberg, D. and C.~Harris, 1992: Evolutionary dynamics with aggregate shocks.
  \textit{Journal of Economic Theory}, \textbf{57~(2)}, 420--441.

\bibitem[{Fudenberg and Levine(1998)}]{FL98}
Fudenberg, D. and D.~K. Levine, 1998: \textit{The Theory of Learning in Games},
  Economic learning and social evolution, Vol.~2. The MIT Press, Cambridge, MA.

\bibitem[{Gilboa and Matsui(1991)}]{GM91}
Gilboa, I. and A.~Matsui, 1991: Social stability and equilibrium.
  \textit{Econometrica}, \textbf{59~(3)}, 859--867.

\bibitem[{Hart and Mas-Colell(2003)}]{HMC03}
Hart, S. and A.~Mas-Colell, 2003: Uncoupled dynamics do not lead to {Nash}
  equilibrium. \textit{American Economic Review}, \textbf{93~(5)}, 1830--1836.

\bibitem[{Hofbauer(1995)}]{Hof95}
Hofbauer, J., 1995: Imitation dynamics for games. Tech. rep., University of
  Vienna.

\bibitem[{Hofbauer(1996)}]{Hof96}
Hofbauer, J., 1996: Evolutionary dynamics for bimatrix games: a {Hamiltonian}
  system? \textit{Journal of Mathematical Biology}, \textbf{34}, 675--688.

\bibitem[{Hofbauer and Sandholm(2009)}]{HS09}
Hofbauer, J. and W.~H. Sandholm, 2009: Stable games and their dynamics.
  \textit{Journal of Economic Theory}, \textbf{144}, 1710--1725.

\bibitem[{Hofbauer and Sandholm(2011)}]{HS11}
Hofbauer, J. and W.~H. Sandholm, 2011: Survival of dominated strategies under
  evolutionary dynamics. \textit{Theoretical Economics}, \textbf{6~(3)},
  341--377.

\bibitem[{Hofbauer and Sigmund(1988)}]{HS88}
Hofbauer, J. and K.~Sigmund, 1988: \textit{The Theory of Evolution and
  Dynamical Systems}. Cambridge University Press.

\bibitem[{Hofbauer et~al.(2009)Hofbauer, Sorin, and Viossat}]{HSV09}
Hofbauer, J., S.~Sorin, and Y.~Viossat, 2009: Time average replicator and best
  reply dynamics. \textit{Mathematics of Operations Research}, \textbf{34~(2)},
  263--269.

\bibitem[{Hofbauer and Weibull(1996)}]{HW96}
Hofbauer, J. and J.~W. Weibull, 1996: Evolutionary selection against dominated
  strategies. \textit{Journal of Economic Theory}, \textbf{71}, 558--573.

\bibitem[{Landau and Lifshitz(1976)}]{LL76}
Landau, L.~D. and E.~M. Lifshitz, 1976: Statistical physics. \textit{Course of
  Theoretical Physics}, Pergamon Press, Oxford, Vol.~5.

\bibitem[{Lee(2003)}]{Lee03}
Lee, J.~M., 2003: \textit{Introduction to Smooth Manifolds}. No. 218 in
  Graduate Texts in Mathematics, Springer-Verlag, New York, NY.

\bibitem[{Mertikopoulos and Moustakas(2010)}]{MM10}
Mertikopoulos, P. and A.~L. Moustakas, 2010: The emergence of rational behavior
  in the presence of stochastic perturbations. \textit{The Annals of Applied
  Probability}, \textbf{20~(4)}, 1359--1388.

\bibitem[{Nash(1950)}]{NashThesis}
Nash, J.~F., 1950: Non-cooperative games. Ph.D. thesis, Princeton University.

\bibitem[{Ritzberger and Vogelsberger(1990)}]{RV90}
Ritzberger, K. and K.~Vogelsberger, 1990: The {Nash} field. Tech. rep.,
  Institute for Advanced Studies, Vienna.

\bibitem[{Ritzberger and Weibull(1995)}]{RW95}
Ritzberger, K. and J.~W. Weibull, 1995: Evolutionary selection in normal-form
  games. \textit{Econometrica}, \textbf{63}, 1371--99.

\bibitem[{Rustichini(1999)}]{Rus99}
Rustichini, A., 1999: Optimal properties of stimulus-response learning models.
  \textit{Games and Economic Behavior}, \textbf{29}, 230--244.

\bibitem[{Samuelson(1993)}]{Sam93}
Samuelson, L., 1993: Does evolution eliminate dominated strategies?
  \textit{Frontiers of Game Theory}, MIT Press, Cambridge, MA.

\bibitem[{Samuelson and Zhang(1992)}]{SZ92}
Samuelson, L. and J.~Zhang, 1992: Evolutionary stability in asymmetric games.
  \textit{Journal of Economic Theory}, \textbf{57}, 363--391.

\bibitem[{Sandholm(2010)}]{San10}
Sandholm, W.~H., 2010: \textit{Population Games and Evolutionary Dynamics}.
  Economic learning and social evolution, MIT Press, Cambridge, MA.

\bibitem[{Sandholm et~al.(2008)Sandholm, Dokumac\i, and Lahkar}]{SDL08}
Sandholm, W.~H., E.~Dokumac\i, and R.~Lahkar, 2008: The projection dynamic and
  the replicator dynamic. \textit{Games and Economic Behavior}, \textbf{64},
  666--683.

\bibitem[{Shamma and Arslan(2005)}]{SA05}
Shamma, J.~S. and G.~Arslan, 2005: Dynamic fictitious play, dynamic gradient
  play, and distributed convergence to {Nash} equilibria. \textit{{IEEE} Trans.
  Autom. Control}, \textbf{50~(3)}, 312--327.

\bibitem[{Smith(1984)}]{Smi84}
Smith, M.~J., 1984: The stability of a dynamic model of traffic assignment - an
  application of a method of {Lyapunov}. \textit{Transportation Science},
  \textbf{18}, 245--252.

\bibitem[{Sorin(2009)}]{Sor09}
Sorin, S., 2009: Exponential weight algorithm in continuous time.
  \textit{Mathematical Programming}, \textbf{116~(1)}, 513--528.

\bibitem[{Taylor and Jonker(1978)}]{TJ78}
Taylor, P.~D. and L.~B. Jonker, 1978: Evolutionary stable strategies and game
  dynamics. \textit{Mathematical Biosciences}, \textbf{40~(1-2)}, 145--156.

\bibitem[{Viossat(2011)}]{Vio11}
Viossat, Y., 2011: Deterministic monotone dynamics and dominated strategies,
  \url{http://arxiv.org/pdf/1110.6246v1}.

\bibitem[{Weibull(1995)}]{Wei95}
Weibull, J.~W., 1995: \textit{Evolutionary Game Theory}. MIT Press, Cambridge,
  MA.

\end{thebibliography}
\endgroup

\end{document}